\documentclass[11pt]{amsart}

\usepackage[title,titletoc,toc]{appendix}
\usepackage[T1]{fontenc}
\usepackage[applemac]{inputenc}
\usepackage{amsmath}
\usepackage{amssymb}
\usepackage{amsfonts}
\usepackage[francais,english]{babel}
\usepackage{amsthm}
\usepackage{enumerate}
\usepackage{accents}
\usepackage{stackengine}

\usepackage{pdfsync}

\usepackage{graphicx,color}

\usepackage[all]{xy}

\usepackage[top=3cm, bottom=2cm, left=3cm, right=3cm]{geometry}

\newtheorem{theorem}{Theorem}
\newtheorem{definition}[theorem]{Definition}
\newtheorem{proposition}[theorem]{Proposition}
\newtheorem{lemma}[theorem]{Lemma}
\newtheorem*{lemma*}{Lemma}

\numberwithin{equation}{section}

\newcommand{\RR}{\mathbb{R}}
\newcommand{\R}{\mathbb{R}}
\newcommand{\N}{\mathbb{N}}
\newcommand{\C}{\mathbb{C}}

\newcommand{\resc}{{\bf resc}}

\newcommand{\mF}{{\mathcal F}}

\newcommand{\mur}{\mu_\resc}

\newcommand{\mY}{{\mathcal Y}}

\newcommand{\vt}{\vartheta_{k,N}}

\title[Uniqueness for stationary Keller-Segel models]{Uniqueness of stationary states for singular Keller-Segel type models}

\author{Vincent Calvez}
\address{Unit\'e de Math\'ematiques Pures et Appliqu\'ees, CNRS UMR 5669 and \'equipe-projet INRIA NUMED, \'Ecole Normale Sup\'erieure de Lyon, Lyon, France}
\email{vincent.calvez@ens-lyon.fr}

\author{Jos\'e Antonio Carrillo}
\address{Mathematical Institute, University of Oxford, Oxford OX2 6GG, UK.} 
\email{carrillo@maths.ox.ac.uk}

\author{Franca Hoffmann}
\address{Department of Computing and Mathematical Sciences, California Institute of Technology, 1200 E California Blvd. MC 305-16, Pasadena, CA 91125, USA.}
\email{fkoh@caltech.edu}

\begin{document}

\maketitle

\begin{abstract}
We consider a generalised Keller-Segel model with non-linear porous medium type diffusion and non-local attractive power law interaction, focusing on potentials that are more singular than Newtonian interaction. We show uniqueness of stationary states (if they exist) in any dimension both in the diffusion-dominated regime and in the fair-competition regime when attraction and repulsion are in balance. As stationary states are radially symmetric decreasing, the question of uniqueness reduces to the radial setting. Our key result is a sharp generalised Hardy-Littlewood-Sobolev type functional inequality in the radial setting.
\end{abstract}


\medskip
\noindent
{\footnotesize
}

     \section{Introduction}\label{sec:intro}

We consider a family of partial differential equations modelling self-attracting
diffusive particles at the macroscopic scale,
\begin{equation}
\left\{
\begin{array}{l}
\partial_t \rho  =  \Delta\rho ^m +
 \nabla\cdot \left( \rho \nabla S\right)\,
, \quad t>0\, , \quad x\in \RR^N\, , \smallskip\\
\displaystyle \rho(t=0,x) = \rho_0(x)\,,
\end{array}
\right. \label{eq:KS}
\end{equation}
where the diffusion exponent $m>1$ is of porous medium type \cite{PME}. 
Since equation \eqref{eq:KS} is positivity preserving, conserves mass, and is invariant by translation, we impose
$$
\rho_0(x)\geq 0\, ,\qquad
\int_{\RR^N}\rho_0(x)\, dx = M\, ,
\qquad \int_{\RR^N}x\rho_0(x)\, dx = 0
$$
for some fixed mass $M>0$, and it follows that the same holds true for the solution $\rho(t,x)$.
The mean-field potential $S(x) := W(x)*\rho(x)$ depends non-locally on the solution $\rho(t,x)$ through convolution with the interaction potential $W(x)$. Depending on the context and the application, different choices of repulsive or attractive potentials are used to model pair-wise interactions between particles, see for instance \cite{BCLR,CaVa2015,CCHCetraro,CCY19} and the references therein. Here, we focus on attractive singular power-law potentials $W(x)=W_k(x)$, 
\begin{equation*}
    W_k(x):=\frac{|x|^k}{k}\,,
    \qquad k<0\,.
\end{equation*}
For $W_k\in L_{loc}^1(\R^N)$, we require $k>-N$. Whilst for $k>1-N$, the gradient $\nabla S_k:= \nabla \left(W_k \ast \rho\right)$ is well defined, it becomes a singular integral in the range $-N<k\leq 1-N$, and we thus define it via a Cauchy principal value. Hence, the mean-field potential gradient in equation \eqref{eq:KS} is given by
\begin{equation}\label{gradS}
 \nabla S_k(t,x) :=
 \begin{cases}
  \nabla W_k(x) \ast \rho(t,x)\, ,
  &\text{if} \, \, k>1-N\, , \\[2mm]
  \displaystyle\int_{\RR^N} \nabla W_k (x-y)\left(\rho(t,y)-\rho(t,x)\right)\, dy\, ,
  &\text{if} \, \, -N<k\leq 1-N\, . 
 \end{cases}
\end{equation}
Writing $k=2s-N$ with $s\in \left(0,\frac{N}{2}\right)$, the convolution term $S_k$ is governed by a fractional diffusion process,
$$
c_{N,s}(-\Delta)^s S_k =  \rho\, , \qquad c_{N,s}=(2s-N) \frac{\Gamma\left(\frac{N}{2}-s\right)}{\pi^{N/2}4^s\Gamma(s)}
=\frac{k\Gamma\left(-k/2\right)}{\pi^{N/2}2^{k+N}\Gamma\left(\frac{k+N}{2}\right)}\, ,
$$
and so the system~\eqref{eq:KS} can be interpreted as
\begin{equation*}
\left\{
\begin{array}{l}
\partial_t \rho  =  \Delta\rho ^m +
 \nabla\cdot \left( \rho \nabla \left(-\Delta\right)^{-s}\rho\right)\,
, \quad t>0\, , \quad x\in \RR^N\, , \smallskip\\
\displaystyle \rho(t=0,x) = \rho_0(x)\,.
\end{array}
\right. 
\end{equation*}
Models with this type of non-local interaction have been considered in \cite{CaVa11a,CaVa11b} in the repulsive case.
Since the non-linear diffusion acts as a repulsive force between particles, one expects competing effects between the diffusion term and the non-local attractive forces, which motivates the study of equilibria of the system. For certain choices of parameters $m$ and $k$, diffusion may overcome attraction, and no stationary states for \eqref{eq:KS} exist. In this case, we seek self-similar profiles instead as they are the natural candidates characterizing the long-time behaviour of the system. Self-similar profiles of equation~\eqref{eq:KS} are stationary states of a suitably rescaled aggregation-diffusion equation with an additional confining potential. Combining both the original and rescaled system, we write
\begin{equation}
\partial_t \rho =  \Delta\rho^m +\nabla\cdot \left( \rho \nabla S_k \right)+\mur \nabla\cdot \left(x \rho\right)\,
, \quad t>0\, , \quad x\in \RR^N\, , 
\label{eq:KSall}
\end{equation}
 with $\mur =0$ for original variables, and $\mur=1$ for rescaled variables. For details on the change of variables transforming \eqref{eq:KS} into \eqref{eq:KSall}, see \cite{CCH1}. 

The competing effects of attractive and repulsive forces can also be observed on the level of the energy functional corresponding to equation \eqref{eq:KSall}:
\begin{equation*}\label{eq:functional}
\mF[\rho]=\frac{1}{m-1}\int \rho^m dx + \frac12 \iint W_k(x-y)\rho(x)\rho(y)\, dxdy
+\frac{\mur}{2}\int |x|^2\rho(x)\,dx\,.
\end{equation*}
More precisely, we can write equation \eqref{eq:KSall} as
$$
\partial_t\rho = \nabla \cdot\left(\rho\nabla \frac{\delta \mF}{\delta \rho}\right)\,
$$
where the first variation of $\mF$ is given by
\begin{equation*}
     \frac{\delta \mF}{\delta\rho}[\rho](x)
    =\frac{m}{m-1} \rho^{m-1} + W_k\ast \rho +\mur \frac{|x|^2}{2} \,,
\end{equation*}
and so solutions to \eqref{eq:KSall} are gradient flows in the 2-Wasserstein metric with respect to the energy $\mF$, see \cite{AGS,Villani03}. One simple way to observe the competition between the diffusion and aggregation term in original variables $\mur=0$ is to consider mass-preserving dilations
$$
\rho_\lambda(x):=\lambda^N\rho(\lambda x)\,.
$$
Substituting $\rho_\lambda$ into $\mF$ with $\mur=0$, we see that the two contributions to the energy are homogeneous with different powers,
\begin{equation*}
\mF[\rho_\lambda]=\frac{\lambda^{N(m-1)}}{m-1}\int \rho^m dx+ \frac{\lambda^{-k}}{2} \iint W_k(x-y)\rho(x)\rho(y)\, dxdy\,,
\end{equation*}
and one observes different types of behaviour depending on the relation between the parameters $N$, $m$ and $k$. The energy functional is homogeneous if attraction and repulsion are in balance, so that the two terms of the energy scale with the same power, that is, if $m=m_c$ for
$$
m_c:= 1-\frac{k}{N}\,.
$$
This motivates the definition
of three different regimes: the \emph{diffusion-dominated regime} $m>m_c$, the \emph{fair-competition regime} $m=m_c$, and the \emph{attraction-dominated regime} $0<m<m_c$.
We will here concentrate on the diffusion-dominated and fair-competition regimes, $m\ge m_c$ in the more singular range $-N<k\le 2-N$. For a detailed overview of the different regimes and recent results, see \cite{CCH1}.

Uniqueness of stationary states is not an immediate consequence of the
gradient flow structure, as the energy functional lacks appropriate
convexity properties (in the sense of McCann's displacement convexity).
Nevertheless, we take advantage of the recent results in \cite{CHVY,CHMV} proving
that any stationary solution in the present setting (see below for
details) are radially symmetric decreasing. Our contribution is to
establish uniqueness of the radial stationary state, using its
reformulation as a critical point of the energy functional (as
expected). Indeed, we prove that any radial critical point of the energy
functional is a global minimiser (as if the functional would be convex),
and we control the equality cases. This amounts to estimate precisely
the balance between the convex part (non-linear diffusion) and the
non-convex part (non-local attraction) in order to show that convexity
is strong enough to discard any other critical point than the global
minimum. Our methodology strongly relies on radial symmetry, so
that \cite{CHVY,CHMV} is a prerequisite to our result.


\subsection{Literature Review}
Let us start by summarising the properties of system~\eqref{eq:KSall} that are known in the literature. 

In the case of the fair-competition regime $m=m_c$ in original variables $\mur=0$, a similar critical mass phenomenon occurs as for the classical Keller-Segel model \cite{BlaDoPe06,BCC12,CaCa11} with logarithmic interaction and linear diffusion. More precisely, it was shown in \cite{BCL} that there exists a critical mass $M_c$ in the case of Newtonian interaction $k=2-N$ for which infinitely many stationary states exist. For sub-critical masses $0<M<M_c$, no stationary states exist as diffusion overcomes attraction, but solutions exist globally in time and decay in a self-similar fashion. For super-critical masses $M>M_c$, attraction overcomes diffusion, and solutions cease to exist in finite time. As shown in \cite{CCH1,CCHCetraro}, this dichotomy holds in fact in the full range $-N<k<0$ in the fair competition regime $m=m_c$.

In the case of the fair-competition regime $m=m_c$ in rescaled variables $\mur=1$ and for subcritical masses $0<M<M_c$ , we have existence of a stationary solution $\bar \rho_M$ with mass $\int \bar \rho_M =M$ by \cite[Theorem 2.9]{CCH1}, which corresponds to a self-similar profile for equation~\eqref{eq:KSall} in original variables with $\mur=0$. Uniqueness of this stationary state is known in one-dimension, as well as convergence in 2-Wasserstein distance of solutions under certain assumptions on the transport map between the solution and the stationary state, see \cite{CCHCetraro}. In higher dimensions, uniqueness and convergence results were shown in \cite{Y14} for the special case of the Newtonian interaction kernel $k=2-N$.

As soon as $m>m_c$, we expect regularising effects from the dominating diffusive term. For the diffusion-dominated regime $m>m_c$ in original variables $\mur=0$, uniform $L^\infty$-bounds were obtained in \cite{Sug2,CaCa06} for any initial mass $M>0$ in the case of Newtonian interactions $k=2-N$. Recently, further results on the existence, boundedness and regularity of solutions have been obtained in \cite{YPZhang18}. Moreover, the existence of global minimisers for the energy functional $\mF$ for any mass $M>0$ was shown in \cite{CCV,CS16} for Newtonian interactions $k=2-N$ and for more general interaction kernels $-N<k<0$ in \cite{CHMV}. Minimisers of the energy functional $\mF$ are stationary states of equation~\eqref{eq:KSall} thanks to the gradient flow structure, as long as they are regular enough. As a direct consequence, we obtain existence of stationary states in the above cases. 
Finally, we point out that existence and uniqueness of stationary solutions have also been obtained for $m=2$  under suitable assumptions in the case of integrable attractive interaction potentials in \cite{Kaib}. The uniqueness of the stationary state for $m>m_c$ was shown in one dimension in \cite{CHMV} using optimal transport techniques, and in the Newtonian case $k=2-N$ in any dimension $N\geq 3$ in \cite{KY12} by a dynamical argument. The most general case for general $k$ however has not been answered yet up to now.

For general attractive potentials, the authors in \cite{delgadino2019uniqueness} recently showed that the uniqueness/non-uniqueness criteria are determined by the power of the degenerate diffusion, with the critical power being $m=2$. In the case $m\geq 2$, they show that for any attractive potential the steady state is unique for a fixed mass. In the case $1<m<2$, they constructed examples of smooth attractive potentials, such that there are infinitely many radially decreasing steady states of the same mass. Here, we concentrate on the specific case of homogeneous potentials $W=W_k$. 


\subsection{Main Results}

Our goal here is to extend the results on the uniqueness of stationary states of system \eqref{eq:KSall} to more singular $k$, higher dimensions $N$ and any $m\ge m_c$ by building on the techniques employed in \cite{CCHCetraro}. We will show that in the case of homogeneous potentials stationary states are indeed unique for all $m\ge m_c$ even if $m<2$ in contrast to \cite{delgadino2019uniqueness}.

We begin by making precise our notion of stationary states.
\begin{definition}\label{def:sstates}
Given $\bar\rho\in L^1_+(\R^N)\cap L^\infty(\R^N)$ we call it a \emph{stationary state} for the evolution problem \eqref{eq:KSall} if $\bar\rho^m\in H^1_{loc}(\R^N)$, $\nabla \bar S_k\in L^1_{loc}(\R^N)$ is as in \eqref{gradS} for $\bar S_k:=W_k\ast\bar\rho$, and it satisfies 
\begin{equation*}
    \nabla\bar\rho^m=-\bar\rho\nabla \bar S_k-\mur x\bar\rho\,.
\end{equation*}
\end{definition}

An important point to make is that due to the results in \cite{CHVY,CHMV}, any stationary solution in the sense of Definition~\ref{def:sstates} in all the cases for $m$, $W$ and $\mur$ discussed in the previous paragraphs are radially symmetric decreasing about their centre of mass and compactly supported. 
This means that the question of uniqueness for stationary states is reduced to the radial setting.
To this end, we rewrite  \eqref{eq:KSall} in radial variables, 
\begin{align}\label{eq:KSrescradial}
 &\partial_t \left(r^{N-1}\rho\right) 
 = 
 \partial_r\left(r^{N-1}\partial_r \rho^m\right) 
+ \partial_r\left(r^{N-1}\rho \partial_r \left(W_k \ast \rho(r)\right)\right)
+ \mur\partial_r\left(r^N \rho\right)\, ,
\end{align}
and work completely in the radial setting from now on. 
Let
$$
 \mY_M := \left\{ \rho \in L_+^1\left(\RR^N\right) \cap L^m \left(\RR^N\right) : ||\rho||_1=M\, , \, \int x\rho(x)\, dx=0\,,\, \mur \int |x|^2\rho(x)\,dx<\infty\right\}.
$$
and its radial subset
$$
\mY_M^*=:\left\{\rho \in \mY_M \, : \, \rho^*=\rho\right\}\, ,
$$
where $\rho^*$ denotes the symmetric decreasing rearrangement of $\rho$.

\begin{theorem}[Sharp Functional Inequality]\label{thm:main1}
 Let $N\geq 2$, $k\in (-N,2-N]$ and $m\ge m_c$. If \eqref{eq:KSrescradial} admits a radial stationary density $\bar \rho$ in $\mY_M^*$, then 
 $$
 \mF[\rho]\geq \mF[\bar \rho]\, , \qquad \forall \, \rho \in \mY_M^*\, ,
 $$
 with the equality cases given by $\bar \rho$, and by its dilations if $m=m_c$ and $\mur=0$.
\end{theorem}

From the above, we can deduce the following uniqueness result.
\begin{theorem}[Uniqueness]\label{thm:main2}
 Let $N\geq 3$ and $k\in(-N,2-N]$. 
 \begin{enumerate}
  \item[(i)] If $m>m_c$ and $\mur=0$, then there is at most one stationary state of \eqref{eq:KSall} for any mass $M>0$ and any centre of mass. Moreover, it coincides with the global minimiser for $\mF$ in $\mY_M^*$ for any $M>0$ (up to translations), as long as $m_c<m<m_*$, where
  \begin{equation*}
  m^*:=
  \begin{cases}
  \frac{2-k-N}{1-k-N}\, , \qquad &\text{if} \quad  -N<k<1-N\, , \\
  + \, \infty &\text{if} \quad 1-N\leq k \le 2-N\, .
  \end{cases}
 \end{equation*}
     \item[(ii)] If $m=m_c$ and $\mur=1$, then there exists at most one stationary state to \eqref{eq:KSall} for any $0<M<M_c$ and with zero centre of mass. Moreover, it coincides with the global minimiser for $\mF$ in $\mY_M^*$.
     \item[(iii)] If $m=m_c$ and $\mur=0$, then there exists at most one stationary state (up to dilations and translations) to \eqref{eq:KSall} for the critical mass $M=M_c$. Moreover, it coincides with the global minimiser for $\mF$ in $\mY_{M_c}^*$.
 \end{enumerate}
\end{theorem}

\subsection{Strategy of Proof}
Our main contribution is Theorem~\ref{thm:main2}. This result however is an immediate consequence of Theorem~\ref{thm:main1}, and the main challenge lies in deriving the functional inequality in Theorem~\ref{thm:main1}. Then Theorem~\ref{thm:main2} follows by noting that all radially symmetric stationary states are in fact global minimisers of $\mF$. The existence of the global minimiser in the above ranges has been proven in \cite{CCH1,CHMV}.

Let us comment in a bit more detail on the strategy of proof for the functional inequality in Theorem~\ref{thm:main1}, and the broader principles at play.
\begin{itemize}
\item To obtain the functional inequality, we need to show a lower bound on the energy $\mF[\rho]$ for generic radial functions $\rho\in\mY_M^*$. This lower bound is related to the PDE \eqref{eq:KS} via its equilibrium states as the lower bound is given precisely by $\mF[\bar\rho]$ where $\bar\rho$ is a stationary state of \eqref{eq:KS}. Our first result is to rewrite the energy functional $\mF[\rho]$ for radial functions $\rho$ in terms of Gauss hypergeometric functions.
\item The two main ingredients for the proof of Theorem~\ref{thm:main1} are  (i) a natural characterisation of radial stationary states for the PDE \eqref{eq:KS} (see Lemma~\ref{lem:charsstatessubNewtonian}), and (ii) a relative convexity inequality (see Lemma~\ref{lem:comparison from below}). The proof of the inequality for (ii) is rather involved, and is the reason we are restricted to the upper bound $k\le 2-N$. 
\item In order to compare $\mF[\rho]$ to $\mF[\bar\rho]$ for any radial function $\rho\in\mY_M^*$, we express $\rho$ as the push-forward of $\bar\rho$ by a radial convex function, and derive an expression for $\mF[\rho]$ in terms of the push-forward map and $\bar\rho$.
\item Jensen's inequality and the convexity result in Lemma~\ref{lem:comparison from below} allow to bound from below the interaction term and the potential term in $\mF[\rho]$.
 \item The characterisation of stationary states in Lemma~\ref{lem:charsstatessubNewtonian} allows to express the interaction term of the energy $\mF[\bar\rho]$ and the lower bound of $\mF[\rho]$ in terms of the diffusive term and potential term of the energies respectively. The key here is to apply the convexity estimates and the characterisation of the stationary states in such a way as to reveal a nice structure of this lower bound; it manifests a direct dependence on the parameter regime (diffusion-dominated vs fair-competition).
\item In order to compare with $\mF[\bar\rho]$, we need to remove the dependence on the push forward. This is achieved thanks to another set of estimates depending on the choice of regime. 
\item Finally, we investigate the equality cases of the inequality in Theorem~\ref{thm:main1}, and prove the claimed uniqueness result.
\end{itemize}

\subsection{Outline} 
In Section~\ref{sec:prelim}, we set up the necessary notation and take advantage of the radial symmetry to derive an explicit formula for the mean-field interaction potential in terms of hypergeomtric functions. Relevant results about hypergeometric functions are summarised in Appendix~\ref{sec:hypergeo}.
In Section~\ref{sec:ineq}, we prove a characterisation of radially symmetric stationary states that then allows us to show the functional inequality in Theorem~\ref{thm:main1} using optimal transport tools. The key ingredient for the proof of Theorem~\ref{thm:main1} is a convexity estimate on the radial interaction potential, see Lemma~\ref{lem:comparison from below}. The proof of the convexity estimate is more involved, and postponed to Appendix~\ref{sec:lem6} for the convenience of the reader. 
We conclude with the proof of Theorem~\ref{thm:main2}, which follows directly from the statement of Theorem~\ref{thm:main1}.

\section{Potentials of radial functions}\label{sec:prelim}
In this section, we prove preliminary results for generic radial functions with the goal to rewrite the energy $\mF[\rho]$ in terms of Gauss hypergeometric functions in the case of radial $\rho$. This will allow us to find a lower bound on the energy in Section~\ref{sec:ineq}. The main conclusion of this section is the following proposition:

\begin{proposition}\label{prop:sec2}
Let $\rho\in\mY_M^*$. Define
\begin{equation*}
\omega(r)=
\begin{cases}
 \frac{1}{2-N}M_\rho(r) r^{2-N},
&\text{ if } k= 2-N\,,\\
\int_{s = 0}^r \frac{r^k}{k} \vt\left(\dfrac s r\right)\rho(s) s^{N-1}\, ds,
&\text{ if } k\neq 2-N\,,
\end{cases}
\end{equation*}
where we denote by $M_\rho(\cdot)$ the cumulative mass of $\rho$ inside balls,
\begin{equation*}
M_\rho(r) = \sigma_N \int_{s = 0}^{r}  \rho(s) s^{N-1}\,ds\, ,
\end{equation*}
and where
\begin{equation}\label{eq:def function H}
\vt(s) = d_N\, F\left( -\dfrac k2,1-\dfrac{k+N}2; \dfrac N2;s^2  \right)\, ,\qquad d_N:=  2^{N-2}\sigma_{N-1}
 \frac{\Gamma\left(\frac{N-1}{2}\right)^2}{\Gamma(N-1)}
\end{equation}
for the Gauss hypergeometric function $F$ defined in \eqref{hyposeries} in the appendix, and with  $\sigma_N=2 \pi^{(N/2)}/\Gamma(N/2)$ denoting the surface area of the $N$-dimensional unit ball.
Then the expression for the energy functional $\mF$ in radial coordinates is given by
\begin{align}\label{eq:Fsubnewtonian}
 \mF[\rho]
  =&\frac{\sigma_N}{m-1}\int_{r=0}^{\infty} \rho(r)^m r^{N-1}\, dr
  + \sigma_N\int_{r = 0}^{\infty}\omega(r)\rho(r)r^{N-1}s^{N-1}dr\\
   &\quad +\mur \frac{\sigma_N}{2} \int_{r=0}^\infty r^2 \rho(r)r^{N-1}\,dr\,.\notag
\end{align}
\end{proposition}

For any radial function $\rho: \RR^N \to \RR$, still denoting by $\rho$ the radial profile of $\rho$, we can write
$$
\int_{\RR^N} \rho(x)\,dx 
= \int_0^{\infty} \int_0^{2\pi} \rho(r) r^{N-1} \sin^{N-2}(\theta)\,d\theta dr
= \sigma_N \int_0^{\infty} \rho(r) r^{N-1} \, dr\, .
$$
Further, if $\delta\ge 1$, then $(\cdot)^\delta$ is convex on the interval $[a,b]$, and by Jensen's inequality 
\begin{equation}\label{eq:weightedJensen}
 \left(\int_a^b \rho(r)\frac{Nr^{N-1}\,dr}{b^N-a^N}\right)^\delta
\leq
\int_a^b \rho(r)^\delta \frac{Nr^{N-1}\,dr}{b^N-a^N}
\end{equation}
with equality if and only if $\rho$ is constant on $[a,b]$.

In the following two lemmata, we derive expressions for the interaction term in the energy $\mF$ for generic radial functions $\rho$ using polar coordinates that will be useful both for the proof of Proposition~\ref{prop:sec2}, and in the sequel.

\begin{lemma}\label{lem:Riesz1}
For $N\geq 2$ and a given radial function $\rho:\RR^N\to \RR$, we have for $|x|=r$
\begin{align*}
|x|^k\ast \rho(x)
&= 2^{N-2} {\sigma_{N-1}}\int_0^\infty (r+\eta)^k H\left(A,B;C;\frac{4r\eta}{(r+\eta)^2}\right)\rho(\eta)\eta^{N-1}\, d\eta
\end{align*}
 where $H(A,B;C; \cdot)$ is given in terms of hypergeometric functions as defined in \eqref{defH} with
 $$
 A=-\frac{k}{2}\, , \qquad B=\frac{N-1}{2}\, , \qquad C=N-1\, .
 $$
\end{lemma}

\begin{proof}
We compute as in \cite[Theorem 5]{TS}, see also \cite{BCLR}, \cite{Dong}  or \cite[\S 1.3]{Dre},
\begin{equation}\label{radialrepresentation}
|x|^k\ast\rho(x)
= {\sigma_{N-1}} \int_0^{\infty} \left(\int_0^\pi
\left(|x|^2+\eta^2-2|x|\eta\cos\theta\right)^{k/2}\,\sin^{N-2}\!\theta\,d\theta\right)\,\rho(\eta) \eta^{N-1}\,d\eta\, .
\end{equation}
Let us define
\begin{align*}\label{esti1}
\Theta_{k,N}(r,\eta)&:= {\sigma_{N-1}}  \int_0^\pi \left(r^2+\eta^2-2r\eta \cos(\theta)\right)^{k/2} \sin^{N-2}(\theta)\, d\theta
=
\begin{cases}
 r^k \vartheta_{k,N}\left(\eta/r\right), &\eta<r\, ,\\
 \eta^k \vartheta_{k,N}\left(r/\eta\right), &r<\eta\, ,
\end{cases}
\end{align*}
  where, for $u\in[0,1)$,
\begin{align*}
\vartheta_{k,N}(u)&:= {\sigma_{N-1}}  \int_0^\pi \left(1+u^2-2u \cos(\theta)\right)^{k/2} \sin^{N-2}(\theta)\, d\theta\\
&=  {\sigma_{N-1}}  \left(1+u\right)^k\int_0^\pi \left(1-4\frac{u}{(1+u)^2} \cos^2\left(\frac{\theta}{2}\right)\right)^{k/2}
\sin^{N-2}(\theta)\, d\theta\, .
\end{align*}
Using the change of variables $t=\cos^2\left(\frac{\theta}{2}\right)$, we get from the integral formulation of hypergeometric functions \eqref{defH} ,
\begin{align*}
 \vartheta_{k,N}(u)&=2^{N-2} {\sigma_{N-1}}  \left(1+u\right)^k\int_0^1 \left(1-\frac{4u}{(1+u)^2}t\right)^{k/2} t^{\frac{N-3}{2}}
 \left(1-t\right)^{\frac{N-3}{2}}\, dt\notag\\
 &= 2^{N-2} {\sigma_{N-1}}  \left(1+u\right)^k H\left(A,B;C;4u/(1+u)^2\right)\,.
\end{align*}
\end{proof}


In order to prove a lower bound on the energy $\mF$ (see Section~\ref{sec:ineq}), our goal is to extend the techniques in \cite{CCHCetraro} to higher dimensions in the case of more singular interaction kernels $-N<k \le 2-N$. For this purpose we need to rewrite the interaction term of the functional even further. Here, we will make use of the formulation in terms of hypergeometric functions as introduced in Lemma \ref{lem:Riesz1}.

\begin{lemma} Let $N\geq 2$ and $k>-N$. For a given radial function $\rho:\RR^N\to \RR$, the attractive mean-field potential rewrites as follows for $|x|=r$:
\begin{equation}
|x|^k*\rho(x) = r^k \int_{\eta = 0}^{r} \vt\left(\dfrac \eta r\right) \rho(\eta) \eta^{N-1}\, d\eta +  \int_{\eta = r}^{\infty} \eta^k \vt\left(\dfrac r \eta\right) \rho(\eta)\eta^{N-1}\, d\eta\,, \label{eq:W*rho}
\end{equation}
where
\begin{equation}\label{eq:def function H}
\vt(s) = d_N\, F\left( -\dfrac k2,1-\dfrac{k+N}2; \dfrac N2;s^2  \right)\, ,\qquad d_N:=  2^{N-2}\sigma_{N-1}
 \frac{\Gamma(B)\Gamma(C-B)}{\Gamma(C)}\, .\end{equation}
with constants $B,C$ as given in Lemma~\ref{lem:Riesz1}.
\end{lemma}
\begin{proof}
As in the proof of Lemma~\ref{lem:Riesz1},
\begin{align*}
|x|^k\ast \rho(x)
=&\int_0^{\infty} \Theta_{k,N}(r,\eta)\rho(\eta) \eta^{N-1}\,d\eta\\
=&\int_0^{r} r^k \vartheta_{k,N}\left(\frac{\eta}{r}\right)\rho(\eta) \eta^{N-1}\,d\eta
+ \int_r^{\infty} \eta^k \vartheta_{k,N}\left(\frac{r}{\eta}\right)\rho(\eta) \eta^{N-1}\,d\eta\, ,
\end{align*}
and by \eqref{defH}, $\vt$ can be written as
\begin{align*}
 \vartheta_{k,N}(s)
 =d_N \left(1+s\right)^k F(A,B;C;4s/(1+s)^2)\, , 
 \qquad d_N:=  2^{N-2}\sigma_{N-1}
 \frac{\Gamma(B)\Gamma(C-B)}{\Gamma(C)}\, ,
\end{align*}
 with
 $$
 A=-\frac{k}{2}\, , \qquad B=\frac{N-1}{2}\, , \qquad C=N-1\,.
 $$
 Using the quadratic transformation \eqref{quadratictrans}, we have for $s \in (0,1)$
\begin{equation*}
F\left( - \dfrac k2,\frac{N-1}{2};N-1; \dfrac{4 s}{\left( 1+s\right)^2}\right) = \left(1 + s\right)^{-k}F\left( -\dfrac k2,1-\dfrac {k+N}2;\dfrac N2;s^2 \right)\, ,
\end{equation*}
and so \eqref{eq:W*rho}-\eqref{eq:def function H} follow.
\end{proof}

The previous two lemmata allow to rewrite the energy $\mF$ in the desired form for any function $\rho\in\mY_M^*$.

\begin{proof}[Proof of Proposition~\ref{prop:sec2}]
We begin with the case $k\neq 2-N$.
Using \eqref{eq:W*rho}, the free energy becomes
\begin{align*}
 \mF[\rho]
 =&\frac{\sigma_N}{m-1}\int_{r=0}^{\infty} \rho(r)^m r^{N-1}\, dr
 + \frac{\sigma_N}{2} \int_{r=0}^{\infty} \rho(r) (W_k\ast \rho)(r) r^{N-1}\, dr\notag\\
 &\quad +\mur \frac{\sigma_N}{2} \int_{r=0}^\infty r^2 \rho(r)r^{N-1}\,dr\notag\\
 =&\frac{\sigma_N}{m-1}\int_{r=0}^{\infty} \rho(r)^m r^{N-1}\, dr
 + \frac{\sigma_N}{2} \int_{r=0}^{\infty}\int_{s = 0}^{r} 
  \frac{r^k}{k} \vt\left(\dfrac s r\right) \rho(r)\rho(s) r^{N-1}s^{N-1}\, dsdr\notag\\
 &\quad+ \frac{\sigma_N}{2} \int_{r=0}^{\infty}\int_{s = r}^{\infty} 
  \frac{s^k}{k} \vt\left(\dfrac r s\right) \rho(r)\rho(s) r^{N-1}s^{N-1}\, dsdr\notag\\
   &\quad +\mur \frac{\sigma_N}{2} \int_{r=0}^\infty r^2 \rho(r)r^{N-1}\,dr\notag\\
  =&\frac{\sigma_N}{m-1}\int_{r=0}^{\infty} \rho(r)^m r^{N-1}\, dr
  + \sigma_N\int_{r = 0}^{\infty}\int_{s = 0}^r \frac{r^k}{k} \vt\left(\dfrac s r\right)  \rho(r)\rho(s) r^{N-1}s^{N-1}\, dsdr\\
   &\quad +\mur \frac{\sigma_N}{2} \int_{r=0}^\infty r^2 \rho(r)r^{N-1}\,dr\,,\notag
\end{align*}
where we swapped the order of integration in the second last line.\\

In the Newtonian case $k=2-N$, we can simplify further. 
By Newton's Shell theorem \cite{LiebLoss},
$$
\partial_r \left(\frac{r^{2-N}}{2-N} \ast \rho\right)(r) = r^{1-N} M_\rho(r)\, .
$$
We have (cf. the Newton's Theorem, \cite[Theorem 9.7]{LiebLoss}):
\begin{equation*}
 \sigma_{N-1} \int_{\theta = 0}^{\pi} \left|r^2 + \eta^2 - 2r\eta\cos(\theta)\right|^{(2-N)/2} \sin(\theta)^{N-2}\, d\theta =  \sigma_N (r\vee \eta )^{2-N}\,, \quad r\vee 
\eta = \max(r,\eta).
\end{equation*}
Therefore, the interaction term of the energy simplifies and using the expression \eqref{radialrepresentation} from the proof of Lemma~\ref{lem:Riesz1}, we obtain
\begin{align*}
 &\iint_{\RR^N\times\RR^N} |x-y|^{2-N}\rho(x)\rho(y)\, dxdy
 = \sigma_N \int_{r=0}^{\infty} \left( \int_{\eta=0}^{\infty} \sigma_N  (r\vee \eta )^{2-N} \rho(\eta)\eta^{N-1}\, d\eta\right) \rho(r) r^{N-1}\, dr\\
 &\quad =\sigma_N \int_{r=0}^{\infty} \left( \sigma_N r^{2-N}\int_{\eta=0}^r \rho(\eta)\eta^{N-1}\, d\eta + \sigma_N \int_{\eta=r}^{\infty} \rho(\eta)\eta\, d\eta \right) \rho(r) r^{N-1}\, dr\\
 &\quad =2 \sigma_N \int_{r=0}^{\infty} \rho(r) M_\rho(r) r\, dr
\end{align*}
by changing the domain of integration in the second term. This concludes the proof of Proposition~\ref{prop:sec2}.
\end{proof}

\section{Functional Inequality}\label{sec:ineq}
In this section, we will prove our main result Theorem~\ref{thm:main1}, from which Theorem~\ref{thm:main2} then immediately follows. To obtain the functional inequality, we need to show a lower bound on the energy $\mF[\rho]$ for generic radial functions $\rho\in\mY_M^*$. This lower bound is related to the PDE \eqref{eq:KS} via its equilibrium states. More precisely, the lower bound of the energy is given by $\mF$ evaluated at the stationary states of \eqref{eq:KS}.  Based on this connection between the energy and the PDE, the two key ingredients in our proof of Theorem~\ref{thm:main1} are (i) a useful characterisation of radial stationary states for the PDE \eqref{eq:KS} (see Lemma~\ref{lem:charsstatessubNewtonian}), and (ii) a relative convexity inequality (see Lemma~\ref{lem:comparison from below}). The proof of the inequality in (ii) is rather involved, and is the reason we are restricted to the upper bound $k<2-N$. As we are not aware of any suitable inequality involving hypergeometric functions, we argue directly from the representation of hypergeometric functions using series, and postpone the proof to Appendix~\ref{sec:lem6}.

 Consider a stationary state $\bar\rho$ according to Definition~\ref{def:sstates}. If $\bar\rho$ is radial, it follows from \eqref{eq:W*rho} that $\bar\rho$ solves 
\begin{align*} 
0
 &= \partial_r \bar\rho^m
+\rho\partial_r \left( \frac{r^k}{k} \int_{s = 0}^{r} \vt\left(\dfrac s r\right) \bar\rho(s) s^{N-1}\, ds +  \int_{s = r}^{\infty} \frac{s^k}{k} \vt\left(\dfrac r s\right) \rho(s)s^{N-1}  \, ds\right)
+\mur r\bar\rho
\end{align*}
if $k\neq 2-N$, and 
\begin{equation*}\label{eq:KSnewton}
0=\partial_r \rho^m +\bar\rho M_{\bar\rho}
 +\mur r\bar\rho
\end{equation*}
in the harmonic case $k=2-N$.
In the sequel we drop the indices in the notation $\vt$ for simplicity, and we write
$$
d\bar r := \bar\rho(r)r^{N-1}\,dr\,.
$$

\begin{lemma}[Characterisation of steady states]\label{lem:charsstatessubNewtonian}
Let $N\geq 2$ and $k>-N$. If $k\neq 2-N$, then any radial stationary state can be written in the form 
    \begin{align}
 \bar \rho(r)^m &=   \int_{ s=r}^{\infty}   \int_{t = 0}^s \left( s^{k-N} \vartheta\left(\dfrac t s\right) - \frac{s^{k-1-N}}{k} t \vartheta'\left(\dfrac t s\right)  \right)  \,d\bar t d\bar s
\notag
\\ &\qquad +  \int_{ s=r}^{\infty}  \int_{t=s}^{\infty}  s^{1-N}\frac{t^{k-1}}{k} \vartheta'\left(\dfrac s t\right) \,d\bar t d\bar s
+\mur \int_{s=r}^\infty s^{2-N}\, d\bar s\,.
\label{eq:StSt characterization radial k}
\end{align}
In the case $k=2-N$, the $\bar\rho$ satisfies
\begin{equation} \label{eq:StStcharacterization radial}
 \bar \rho^m(r)  =  \int_{s=r}^{\infty}  M_{\bar\rho}(s)   s^{2-2N} \, d\bar s  
 +\mur \int_{s=r}^\infty s^{2-N}\, d\bar s\, .
\end{equation}
\end{lemma}

\begin{proof}
By Definition~\ref{def:sstates}, any stationary state $\bar \rho$ satisfies
\begin{align*}
-   \dfrac d{dr} \bar \rho(r)^m  & =  \bar \rho(r)  \dfrac d{dr}\left(W_k*\bar \rho\right)(r) +\mur r \bar\rho(r)
\, . 
\end{align*}
Integrating between $r$ and $\infty$, we deduce that 
\begin{equation*}
 \bar \rho(r)^m = \int_{ s = r}^{\infty} \bar \rho(s) \dfrac d{ds}\left(W_k*\bar \rho\right)(s)\, ds
 +\mur \int_{s=r}^\infty s^{2-N}\bar\rho(s)s^{N-1}\,ds
 \, .
\end{equation*}
It remains to examine the term $(d/ds)(W_k*\bar \rho)(s)$. 
We differentiate the expression \eqref{eq:W*rho} to obtain
\begin{align*}
 \dfrac d{ds}\left(W_k*\bar \rho\right)(s)
&=   
 \int_{t = 0}^s \left(  s^{k-1} \vartheta\left(\dfrac t s\right) - \frac{s^{k-2}}{k} t \vartheta'\left(\dfrac t s\right)  \right)   \bar \rho(t)t^{N-1}\, dt
+  \int_{t=s}^{\infty} \frac{t^{k-1}}{k} \vartheta'\left(\dfrac s t\right) \bar \rho(t)t^{N-1} \,  dt\, .
\end{align*}
This yields the claimed characterisation.
For $k=2-N$, the result follows directly from Newton's Shell Theorem, see the end of Section~\ref{sec:prelim}.
\end{proof}

\begin{figure}
\begin{center}
\includegraphics[width = .48\linewidth]{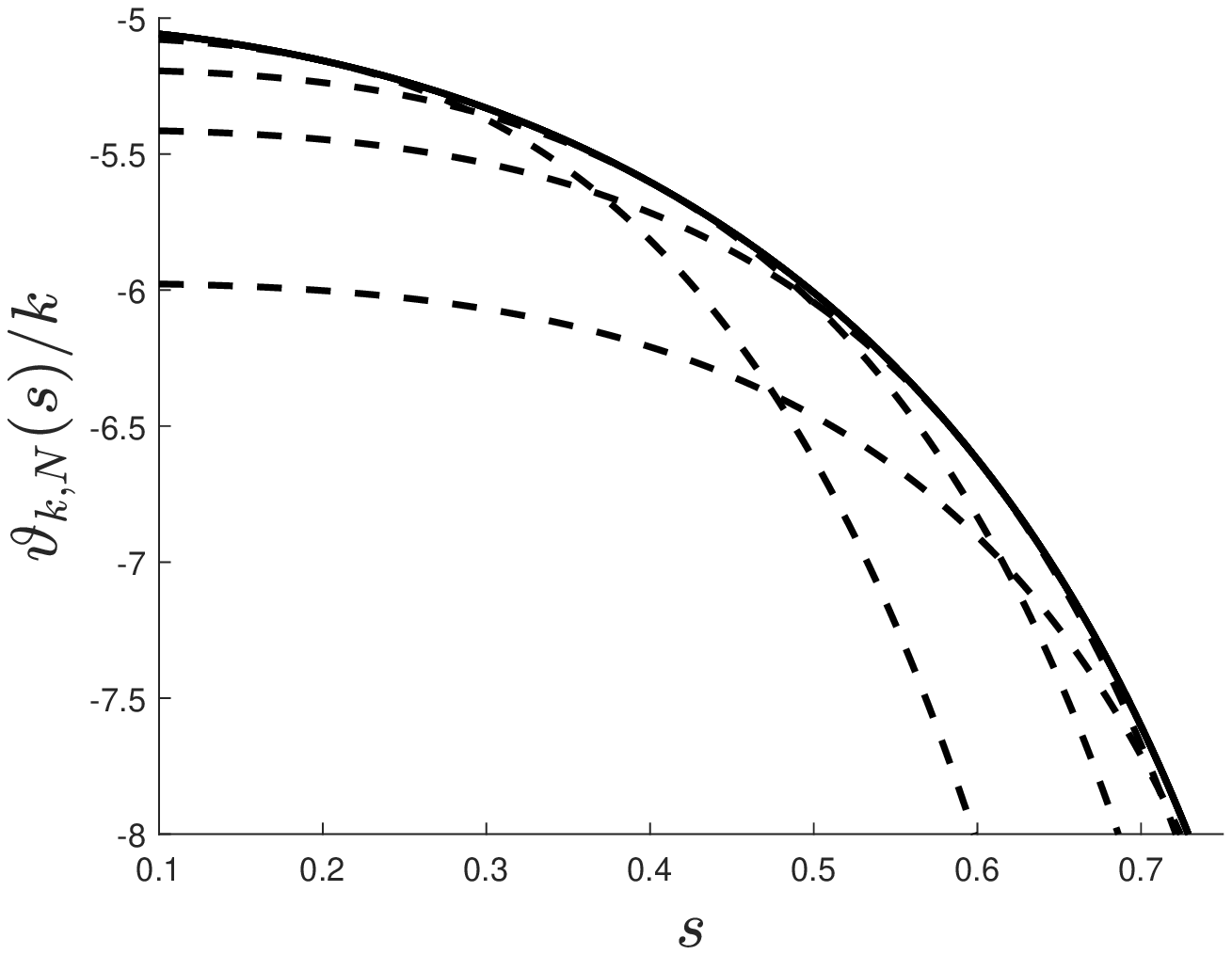} \,
\includegraphics[width = .48\linewidth]{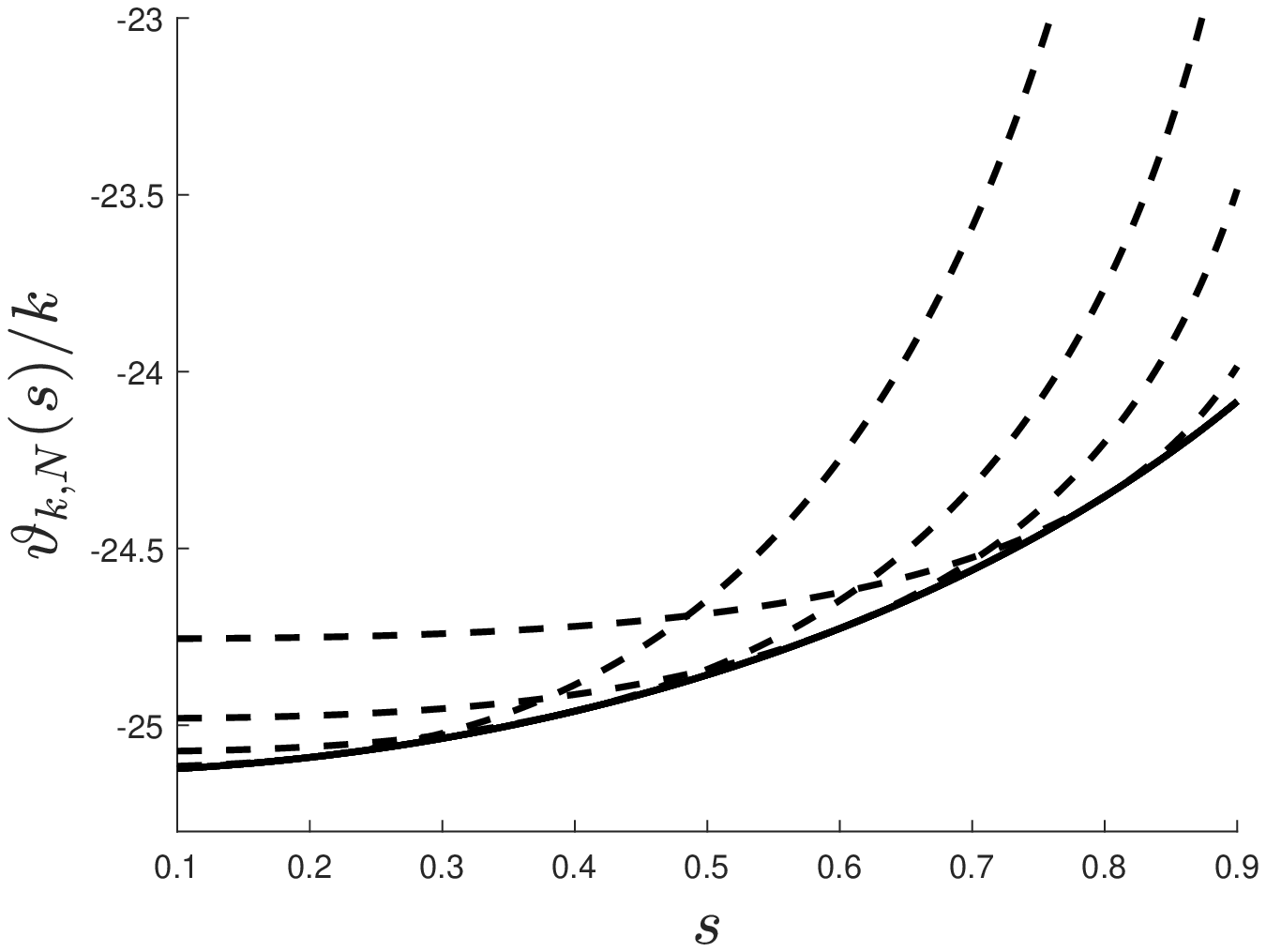} \\
\caption{Numerical illustration of Lemma \ref{lem:comparison from below} for $N = 3$, and (left) $k = -2.5$ or (right) $k = -0.5$ for tangents $c=0.2, 0.4, 0.6, 0.8$ (dotted lines). For $k<2-N$ where convexity holds all tangents lie below the curve $\vt(s)/k$ (black line) as shown in Lemma \ref{lem:comparison from below}, which does not hold for $k>2-N$.}
\label{fig:convexity}
\end{center}
\end{figure}

It follows from the above characterisation that for any function $g:[0,\infty)\to\RR$, a stationary state of \eqref{eq:KSrescradial} satisfies
\begin{align}\label{eq:SSg}
    \int_{a=0}^\infty g(a)\bar\rho^m(a)a^{N-1}\,da
    &= \int_{a=0}^\infty\int_{b=a}^\infty\int_{s=0}^b g(a)a^{N-1} \left[ b^{k-N} \vartheta\left(\dfrac s b\right) - \frac{b^{k-1-N}}{k} s \vartheta'\left(\dfrac s b\right)  \right]  \,d\bar s d\bar b da\notag\\
    &\quad+ \int_{a=0}^\infty\int_{b=a}^\infty\int_{s=b}^\infty g(a)a^{N-1}b^{1-N} \frac{s^{k-1}}{k}\vartheta'\left(\dfrac b s\right) \,d\bar s d\bar b da\\
    &\quad + \mur \int_{a=0}^\infty\int_{b=a}^\infty g(a)a^{N-1} b^{2-N} \,d\bar b da\,.\notag
\end{align}
This expression will be useful for proving the functional inequality in Theorem~\ref{thm:main1}.
Moreover, in order to prove Theorem~\ref{thm:main1}, we seek an inequality of the following type:
\begin{equation*} 
\frac{\vt(s)}{k} \geq \alpha + \beta (1 - s^N)^{k/N}\, . 
\end{equation*}
The constants $\alpha$ and $\beta$ are chosen so that the above inequality is an equality at zero and first order for a convenient choice of $s$ (to be chosen later). This writes into the following lemma:

\begin{lemma} \label{lem:comparison from below}
Assume $N\geq 2$ and $k\in (-N,2-N)$. The following inequality holds true for any $(s,c)\in(0,1)^2$:
\begin{equation} \label{eq:comparison from below}
\frac{\vt(s)}{k}\geq \alpha(c) +\beta(c) \left(1 - s^N\right)^{k/N}\, .
\end{equation}
with the two factors given by
\[\alpha(c):=\frac{\vt(c)}{k} + \frac{1}{k^2} c^{1-N} \left(1 - c^N\right) \vt'(c) \leq 0\, ,\]
and
\[\beta(c):=-\frac{1}{k^2} c^{1-N} \left(1 - c^N\right)^{1 - k/N} \vt'(c) \leq 0\, .\]
\end{lemma}
Note that this crucial lemma is the reason we are restricted to the upper bound $2-N$ in $k$, see Figure~\ref{fig:convexity}. The proof is postponed to the Appendix \ref{sec:lem6} due to technicality. 

We are now ready to prove our main result.

\begin{proof}[Proof of Theorem~\ref{thm:main1} for $k< 2-N$]
We begin with the more complicated case $k\in (-N,2-N)$ as the harmonic case $k=2-N$ will follow in a similar manner. We break down the argument into seven steps, and briefly summarize our strategy here: 
\begin{enumerate}
\item Using the radial expression of the energy in Proposition~\ref{prop:sec2} and the characterisation of stationary states in Lemma~\ref{lem:charsstatessubNewtonian}, we derive an expression for the energy $\mF[\bar\rho]$ of the stationary state $\bar\rho$. The key here is to express the interaction part of the energy in terms of the diffusion + confining potential thanks to Lemma~\ref{lem:charsstatessubNewtonian}.
\item In order to compare $\mF[\rho]$ to $\mF[\bar\rho]$ for any radial function $\rho\in\mY_M^*$, we express $\rho$ as the push-forward of $\bar\rho$ by a radial convex function. Using again the radial formulation of the energy in Proposition~\ref{prop:sec2}, we are thus able to derive an expression for $\mF[\rho]$ in terms of the push-forward map and $\bar\rho$.
\item  In Step 3, we use Jensen's inequality and the convexity result in Lemma~\ref{lem:comparison from below} to obtain a lower bound on the interaction term in $\mF[\rho]$ as derived in Step 2.
\item Similar to Step 1 for $\mF[\bar\rho]$, we now use the characterisation of stationary states in Lemma~\ref{lem:charsstatessubNewtonian} to reformulate the lower bound on the interaction term in $\mF[\rho]$ derived in Step 3 in terms of the diffusion + confining potential.
\item In Step 5, firstly, we use convexity estimates for the confining potential term in $\mF[\rho]$ to bound it from below, similar to how we used Jensen's inequality in Step 3 to obtain a lower bound on the interaction term. Secondly, we combine the lower bound on the interaction term derived in Step 4 and the lower bound on the confining potential to obtain an overall lower bound on the energy $\mF[\rho]$. The convexity estimates and the characterisation of the stationary state in Steps 1-5 are applied in such a way as to reveal a nice structure of this lower bound in terms of the choice of parameters $(m,k)$. In particular, it reveals how the lower bound depends on the choice of regime (diffusion-dominated vs fair-competition).
\item The lower bound in Step 5 depends on $(m,k)$, the push foward map, and the stationary state $\bar\rho$. In order to compare with $\mF[\bar\rho]$, we need to remove the dependence on the push forward. This is achieved thanks to another set of estimates depending on the choice of regime. This concludes the proof of the inequality as stated in Theorem~\ref{thm:main1}.
\item In a final step, we investigate the equality cases of the inequality derived in Step 6, and prove the claimed uniqueness result.
\end{enumerate}

\noindent\textbf{Step 1 (Energy of stationary state)}: Following Proposition~\ref{prop:sec2}, the energy of the stationary state $\bar \rho$ is given by
\begin{align*}
 \frac{1}{N\sigma_N}\mF[\bar \rho]
 =& \dfrac{1}{N(m-1)}\int_{a=0}^{\infty} \bar\rho(a)^{m}\, a^{N-1} da 
 +   \frac{1}{Nk} \int_{a = 0}^{\infty} \int_{b=0}^a a^{k} \vartheta\left(\frac{b}{a}\right) \,d\bar b d\bar a
 + \frac{\mur}{2N} \int_{a=0}^\infty a^2\,d\bar a \, .
 \end{align*}
Choosing $g=id$ in \eqref{eq:SSg} and rewriting the domain of integration, we obtain
\begin{align*}
 &\int_{a=0}^{\infty} \bar \rho(a)^m a^{N-1}\, da\\
 &= \int_{a=0}^{\infty}  \int_{ b=a}^{\infty}   \int_{s = 0}^b \left[  b^{k-N} \vartheta\left(\dfrac s b\right) - \frac{b^{k-1-N}}{k} s \vartheta'\left(\dfrac s b\right)  \right] a^{N-1} \,d\bar s d\bar b  da\\  
 &\quad+ \frac{1}{k} \int_{a=0}^{\infty}\int_{ b=a}^{\infty}  \int_{s=b}^{\infty}  b^{1-N}s^{k-1} \vartheta'\left(\dfrac b s\right) a^{N-1} \,d\bar s d\bar b  da
+\mur \int_{a=0}^\infty\int_{b=a}^\infty a^{N-1}b^{2-N} \,d\bar b da\\
&=\int_{b=0}^{\infty}  \int_{ s=0}^{b}   \left(\int_{a = 0}^ba^{N-1}\,da\right) \left[  b^{k-N} \vartheta\left(\dfrac s b\right) - \frac{b^{k-1-N}}{k} s \vartheta'\left(\dfrac s b\right)  \right]  \,d\bar s d\bar b  
\\  &\quad+  \frac{1}{k} \int_{b=0}^{\infty}\int_{s=b}^{\infty}  \left(\int_{a=0}^{b} a^{N-1}\, da\right)  b^{1-N}s^{k-1} \vartheta'\left(\dfrac b s\right)  \,d\bar s d\bar b  
+\mur \int_{b=0}^\infty\left(\int_{a=0}^b a^{N-1}\, da\right)b^{2-N} \,d\bar b \\
&=\frac{1}{N}\left(\int_{b=0}^{\infty}  \int_{ s=0}^{b}  b^k\vartheta\left(\dfrac s b\right)   \,d\bar s d\bar b 
+\mur  \int_{b=0}^\infty b^{2} \,d\bar b \right)
\, .
\end{align*}
Hence,
\begin{align*}
&\frac{1}{Nk} \int_{a = 0}^{\infty} \int_{b=0}^a a^{k} \vartheta\left(\frac{b}{a}\right) \,d\bar b d\bar a
= \frac{1}{k}\int_{a=0}^{\infty} \bar \rho(a)^m a^{N-1}\, da
 -\frac{\mur}{kN} \int_{a=0}^\infty a^2\,d\bar a \,,
\end{align*}
and so we conclude
\begin{align}\label{Fsstatetransport}
 \frac{1}{N\sigma_N}\mF[\bar \rho]
 =&  \int_{a=0}^{\infty}\left(\frac{1}{N(m-1)} + \frac{1}{k}\right) \bar\rho(a)^{m}\, a^{N-1} da 
+\frac{\mur}{N}\left(\frac{1}{2} -\frac{1}{k}\right) \int_{a=0}^\infty a^2\,d\bar a \,.
 \end{align}

\noindent\textbf{Step 2 (Write $\mF[\rho]$ in terms of $\bar\rho$)}:
Next, we write the energy $\mF[\rho]$ in terms of $\bar\rho$ and our goal is to find suitable estimates from below. 
For a given stationary state $\bar \rho \in \mY_M^*$ and any radial function $\rho \in \mY_M^*$, we denote by $\psi$ the radial profile of the convex function whose gradient pushes forward the measure $\bar \rho(a)a^{N-1} da$ onto $\rho(r) r^{N-1}dr$: 
$$\psi' \# \left(\bar\rho(a) a^{N-1} da\right) = \rho(r) r^{N-1} dr\,.$$ 
Changing variables $ r = \psi'(a)$, we have
$$
\rho(r)r^{N-1}\,dr=\bar\rho(a)a^{N-1}\,da =:d \bar a
$$
and 
$$
\rho(r)=\dfrac{a^{N-1} \bar \rho(a)}{\psi'(a)^{N-1}\psi''(a)} 
=\frac{\bar\rho(a)}{\varphi(a)}\,,\qquad \text{for } \quad
\varphi(a):=\dfrac{1}{Na ^{N-1}}   \dfrac d{da } \left( \psi'\right)^N(a)\, .
$$
Then the repulsive term of the functional $\mF[\rho]$ rewrites
\begin{align*}
  \dfrac{\sigma_N}{m-1}\int_{r=0}^{\infty} \rho(r)^{m}\, r^{N-1} dr
  =&\dfrac{\sigma_N}{m-1} \int_{a=0}^{\infty}  \varphi(a)^{1-m} \bar \rho(a )^m  a^{N-1}\, da\, ,
\end{align*}
and following \eqref{eq:Fsubnewtonian}in Proposition~\ref{prop:sec2}, the interaction term becomes
\begin{align*}
 &\sigma_N\int_{r = 0}^{\infty}\int_{s = 0}^r \frac{r^k}{k} \vartheta\left(\dfrac s r\right)  \rho(r) \rho(s)r^{N-1}s^{N-1}\, dsdr  
 =\frac{\sigma_N}{k}\int_{a = 0}^{\infty}\int_{b = 0}^a (\psi'(a))^k \vartheta\left(\dfrac {\psi'(b)}{\psi'(a)}\right) \,d\bar bd\bar a\, .
\end{align*}
We therefore have
\begin{align}\label{Ftransport}
   \frac{1}{N\sigma_N}\mF[\rho]
   &=\dfrac{1}{N(m-1)} \int_{a=0}^{\infty}  \varphi(a)^{1-m} \bar \rho(a )^m  a^{N-1}\, da\\
   &\quad + \frac{1}{Nk}\int_{a = 0}^{\infty}\int_{b = 0}^a (\psi'(a))^k \vartheta\left(\dfrac {\psi'(b)}{\psi'(a)}\right) \,d\bar bd\bar a
   + \frac{\mur}{2N}\int_{a=0}^\infty (\psi'(a))^2\,d\bar a\,.\notag
\end{align}
%

\noindent\textbf{Step 3 (Convexity inequalities for the interaction term)}: In this step, we use Jensen's inequality and the convexity result in Lemma~\ref{lem:comparison from below} to obtain a lower bound on the interaction term of $\mF[\rho]$ as stated in \eqref{Ftransport}.

By Jensen's inequality \eqref{eq:weightedJensen}, we estimate
\begin{align}\label{Jensen1}
    (\psi'(a))^k
    &= a^k\left(\frac{\psi'(a)^N}{a^N}\right)^{k/N}
    =a^k\left(\int_{s=0}^a \varphi(s)\frac{Ns^{N-1}}{a^N}\,ds\right)^{k/N}\\
    &\le N a^{k-N}\int_{s=0}^a \varphi(s)^{k/N}s^{N-1}\,ds\notag\,,
\end{align}
\begin{align}\label{Jensen2}
    \left(\psi'(a)^N-\psi'(b)^N\right)^{k/N}
    &=  \left(\frac{\psi'(a)^N-\psi'(b)^N}{a^N-b^N}\right)^{k/N}(a^N-b^N)^{k/N}\\
    &=(a^N-b^N)^{k/N}\left(\int_{s=b}^a \varphi(s)\frac{Ns^{N-1}}{(a^N-b^N)}\,ds\right)^{k/N}\notag\\
    &\le N (a^N-b^N)^{k/N-1}\int_{s=b}^a \varphi(s)^{k/N}s^{N-1}\,ds\,.\notag
\end{align}
Using first the comparison equation~\eqref{eq:comparison from below} in Lemma~\ref{lem:comparison from below} with $c=b/a$ and then \eqref{Jensen1}--\eqref{Jensen2}, we obtain the estimate
\begin{align*}
&\frac{1}{Nk}\int_{a = 0}^{\infty}\int_{b = 0}^a (\psi'(a))^k\vartheta\left(\dfrac {\psi'(b)}{\psi'(a)}\right) \,d\bar bd\bar a \\
& \geq \frac{1}{N}\int_{a = 0}^{\infty}\int_{b = 0}^a (\psi'(a))^k
\alpha\left(\frac{b}{a}\right) \,d\bar bd\bar a
+ \frac{1}{N}\int_{a = 0}^{\infty}\int_{b = 0}^a \beta\left(\frac{b}{a}\right)
\left( (\psi'(a))^N- (\psi'(b))^N\right)^{k/N}\,d\bar bd\bar a
\\
%
& \geq \int_{a = 0}^{\infty}\int_{b = 0}^a\int_{s=0}^a a^{k-N} \varphi(s)^{k/N} s^{N-1} 
\alpha\left(\frac{b}{a}\right) \,dsd\bar bd\bar a\\
&\quad+ \int_{a = 0}^{\infty}\int_{b = 0}^a \int_{s=b}^a (a^N-b^N)^{k/N-1} \varphi(s)^{k/N} s^{N-1} \beta\left(\frac{b}{a}\right)\,dsd\bar bd\bar a
\\
%
%
&= \frac{1}{k}\int_{a = 0}^{\infty}\int_{b = 0}^a\int_{s=0}^a a^{k-N} \varphi(s)^{k/N} s^{N-1} 
\left( \vartheta\left(\dfrac b a\right) + \dfrac1k \left(\dfrac b a\right)^{1-N} \left(1 - \left(\dfrac b a\right)^N\right)  \vartheta'\left(\dfrac b a\right) \right) 
\,dsd\bar bd\bar a\\
&\quad-\frac{1}{k^2} \int_{a = 0}^{\infty}\int_{b = 0}^a \int_{s=b}^a (a^N-b^N)^{k/N-1} \varphi(s)^{k/N} s^{N-1} 
\left(\dfrac b a\right)^{1-N} \left(1 - \left(\dfrac b a\right)^N\right)^{1 - k/N} \vartheta'\left(\dfrac b a\right) 
\,dsd\bar bd\bar a\\
&=\frac{1}{k} \int_{a = 0}^{\infty}\int_{b = 0}^a\int_{s=0}^a a^{k-N} \varphi(s)^{k/N} s^{N-1}  \vartheta\left(\dfrac b a\right) 
\,dsd\bar bd\bar a\\
&\quad+ \frac{1}{k^2} \int_{a = 0}^{\infty}\int_{b = 0}^a\int_{s=0}^a a^{k-N} \varphi(s)^{k/N} s^{N-1} 
 \left(\dfrac b a\right)^{1-N} \left(1 - \left(\dfrac b a\right)^N\right) \vartheta'\left(\dfrac b a\right)
\,dsd\bar bd\bar a\\
&\quad-\frac{1}{k^2} \int_{a = 0}^{\infty}\int_{b = 0}^a \int_{s=b}^a (a^N-b^N)^{k/N-1} \varphi(s)^{k/N} s^{N-1} 
\left(\dfrac b a\right)^{1-N} \left(1 - \left(\dfrac b a\right)^N\right)^{1 - k/N} \vartheta'\left(\dfrac b a\right) 
\,dsd\bar bd\bar a\,.
\end{align*}
Note that the signs in \eqref{Jensen1}--\eqref{Jensen2} flip since $\alpha\left(\frac{b}{a}\right)$ and $\beta\left(\frac{b}{a}\right)$ are non-positive. Finally, the above simplifies to
\begin{align*}
&\frac{1}{Nk}\int_{a = 0}^{\infty}\int_{b = 0}^a (\psi'(a))^k\vartheta\left(\dfrac {\psi'(b)}{\psi'(a)}\right) \,d\bar bd\bar a \\
&\ge \frac{1}{k} \int_{a = 0}^{\infty}\int_{b = 0}^a\int_{s=0}^a a^{k-N} \varphi(s)^{k/N} s^{N-1}  \vartheta\left(\dfrac b a\right) 
\,dsd\bar bd\bar a\\
&\quad+ \frac{1}{k^2} \int_{a = 0}^{\infty}\int_{b = 0}^a\int_{s=0}^a  \varphi(s)^{k/N} s^{N-1} a^{k-N-1}b^{1-N} \left(a^N - b^N\right) \vartheta'\left(\dfrac b a\right)
\,dsd\bar bd\bar a\\
&\quad-\frac{1}{k^2} \int_{a = 0}^{\infty}\int_{b = 0}^a \int_{s=b}^a \varphi(s)^{k/N} s^{N-1} b^{1-N} a^{k-1} \vartheta'\left(\dfrac b a\right) 
\,dsd\bar bd\bar a\\
&=\frac{1}{k} \int_{a = 0}^{\infty}\int_{b = 0}^a\int_{s=0}^a a^{k-N} \varphi(s)^{k/N} s^{N-1}  \vartheta\left(\dfrac b a\right) 
\,dsd\bar bd\bar a\\
&\quad+ \frac{1}{k^2} \int_{a = 0}^{\infty}\int_{b = 0}^a\int_{s=0}^a  \varphi(s)^{k/N} s^{N-1} a^{k-1}b^{1-N}   \vartheta'\left(\dfrac b a\right)
\,dsd\bar bd\bar a\\
&\quad-\frac{1}{k^2} \int_{a = 0}^{\infty}\int_{b = 0}^a\int_{s=0}^a  \varphi(s)^{k/N} s^{N-1} a^{k-N-1}b \vartheta'\left(\dfrac b a\right)
\,dsd\bar bd\bar a\\
&\quad-\frac{1}{k^2} \int_{a = 0}^{\infty}\int_{b = 0}^a \int_{s=b}^a \varphi(s)^{k/N} s^{N-1} a^{k-1} b^{1-N}  \vartheta'\left(\dfrac b a\right) 
\,dsd\bar bd\bar a\\
&=\frac{1}{k} \int_{a = 0}^{\infty}\int_{b = 0}^a\int_{s=0}^a a^{k-N} \varphi(s)^{k/N} s^{N-1}  \vartheta\left(\dfrac b a\right) 
\,dsd\bar bd\bar a%
\end{align*}
\begin{align*}
&\quad+ \frac{1}{k^2} \int_{a = 0}^{\infty}\int_{b = 0}^a\int_{s=0}^b  \varphi(s)^{k/N} s^{N-1} a^{k-1}b^{1-N}   \vartheta'\left(\dfrac b a\right)
\,dsd\bar bd\bar a\\
&\quad-\frac{1}{k^2} \int_{a = 0}^{\infty}\int_{b = 0}^a\int_{s=0}^a  \varphi(s)^{k/N} s^{N-1} a^{k-N-1}b \vartheta'\left(\dfrac b a\right)
\,dsd\bar bd\bar a\\
&=\frac{1}{k} \int_{a = 0}^{\infty}\int_{b = 0}^a\int_{s=0}^a  \varphi(s)^{k/N} s^{N-1}  \left[a^{k-N}\vartheta\left(\dfrac b a\right) 
-  \frac{a^{k-N-1}}{k}b \vartheta'\left(\dfrac b a\right)\right]
\,dsd\bar bd\bar a\\
&\quad+ \frac{1}{k^2} \int_{a = 0}^{\infty}\int_{b = 0}^a\int_{s=0}^b  \varphi(s)^{k/N} s^{N-1} a^{k-1}b^{1-N}   \vartheta'\left(\dfrac b a\right)
\,dsd\bar bd\bar a\,.
\end{align*}

\noindent\textbf{Step 4 (Characterisation of stationary states)}:
Next, we make use of the characterisation \eqref{eq:SSg} of stationary states for $g(a)=\varphi(a)^{k/N}$ to be able to rewrite the lower bound on the interaction term obtained in Step 3 in such a way that it has a similar structure as the diffusion term and the second moment term in $\mF[\rho]$ as stated at the end of Step 2.

Writing $\int_{a=0}^\infty\int_{s=a}^\infty=\int_{s=0}^\infty\int_{a=0}^s$ and exchanging $a$ and $s$, the expression \eqref{eq:SSg} for stationary states with the choice $g(a)=\varphi(a)^{k/N}$ can be written as
\begin{align*}
   & \int_{a=0}^\infty \varphi(a)^{1-m_c}\bar\rho(a)^m a^{N-1}\,da\\
&\quad=\int_{a=0}^\infty\int_{s=a}^\infty \int_{b=0}^s \varphi(a)^{k/N}a^{N-1} \left[s^{k-N}\vartheta\left(\dfrac b s\right) 
 -\frac{s^{k-N-1}}{k}b \vartheta'\left(\dfrac b s\right)\right]
\,dad\bar bd\bar s\\
&\qquad+ \int_{a = 0}^{\infty}\int_{s=a}^\infty\int_{b=s}^\infty  \varphi(a)^{k/N} a^{N-1} \frac{b^{k-1}}{k}s^{1-N}   \vartheta'\left(\dfrac s b\right)
\,dad\bar bd\bar s
+ \mur \int_{a=0}^\infty\int_{s=a}^\infty \varphi(a)^{k/N} a^{N-1}s^{2-N}\,da d\bar s
%
\end{align*}
\begin{align*}
&\quad=\int_{a=0}^\infty\int_{s=0}^a \int_{b=0}^a \varphi(s)^{k/N}s^{N-1} \left[a^{k-N}\vartheta\left(\dfrac b a\right) 
 -\frac{a^{k-N-1}}{k}b \vartheta'\left(\dfrac b a\right)\right]
\,dsd\bar bd\bar a\\
&\qquad+ \int_{a = 0}^{\infty}\int_{s=0}^a \int_{b=a}^\infty  \varphi(s)^{k/N} s^{N-1} \frac{b^{k-1}}{k}a^{1-N}   \vartheta'\left(\dfrac a b\right)
\,dsd\bar bd\bar a
+ \mur \int_{a=0}^\infty\int_{s=0}^a \varphi(s)^{k/N} s^{N-1}a^{2-N}\,ds d\bar a\,.
\end{align*}
Writing $\int_{a=0}^\infty\int_{b=a}^\infty=\int_{b=0}^\infty\int_{a=0}^b$ in the second term only and exchanging $a$ and $b$, we conclude
\begin{align*}
&\int_{a=0}^\infty\int_{s=0}^a \int_{b=0}^a \varphi(s)^{k/N}s^{N-1} \left[a^{k-N}\vartheta\left(\dfrac b a\right) 
 -\frac{a^{k-N-1}}{k}b \vartheta'\left(\dfrac b a\right)\right]
\,dsd\bar bd\bar a\\
&\qquad+ \int_{a = 0}^{\infty}\int_{s=0}^a \int_{b=0}^a  \varphi(s)^{k/N} s^{N-1} \frac{a^{k-1}}{k}b^{1-N}   \vartheta'\left(\dfrac b a\right)
\,dsd\bar bd\bar a\\
&=\int_{a=0}^\infty \varphi(a)^{1-m_c}\bar\rho(a)^m a^{N-1}\,da
- \mur \int_{a=0}^\infty\int_{s=0}^a \varphi(s)^{k/N} s^{N-1}a^{2-N}\,ds d\bar a\,.
\end{align*}
This expression allows to rewrite the lower bound on the interaction term obtained in Step 3. More precisely, we obtain the following lower bound on the interaction term:
\begin{align*} 
&\frac{1}{Nk}\int_{a = 0}^{\infty}\int_{b = 0}^a (\psi'(a))^k\vartheta\left(\dfrac {\psi'(b)}{\psi'(a)}\right) \,d\bar bd\bar a \\
&\quad\ge\frac{1}{k}\int_{a=0}^\infty \varphi(a)^{1-m_c}\bar\rho(a)^m a^{N-1}\,da
 - \frac{\mur}{k} \int_{a=0}^\infty\int_{s=0}^a \varphi(s)^{k/N} s^{N-1}a^{2-N}\,ds d\bar a\,.\notag
\end{align*}

\noindent\textbf{Step 5 (Lower Bound on $\mF[\rho]$)}: In this step, we first apply convexity estimates on the second moment term in \eqref{Ftransport}, and then combine these bounds with the bound we obtained on the interaction term of the energy in Step 4. This provides a new lower bound on the energy $\mF[\rho]$ revealing a nice structure that depends on the choice of parameter regime $(m,k)$.

We estimate the confinement term in the energy \eqref{Ftransport} (Step 2) using Jensen's inequality \eqref{eq:weightedJensen} as in \eqref{Jensen1} (Step 3):
\begin{align}
      (\psi'(a))^2
    &= a^2\left(\frac{\psi'(a)^N}{a^N}\right)^{2/N}
    =a^2\left(\int_{s=0}^a \varphi(s)\frac{Ns^{N-1}}{a^N}\,ds\right)^{2/N}\notag\\
    &\ge N a^{2-N}\int_{s=0}^a \varphi(s)^{2/N}s^{N-1}\,ds\,.
    \label{confinementJensen}
\end{align}
Substituting these estimates into \eqref{Ftransport}, we obtain
\begin{align*}
   \frac{1}{N\sigma_N}\mF[\rho]
   &\ge \frac1N\int_{a=0}^{\infty} \left(\frac{\varphi(a)^{1-m}}{m-1}  -\frac{\varphi(a)^{1-m_c}}{m_c-1}\right) \bar \rho(a )^m  a^{N-1}\, da\\
   &\quad + \mur\int_{a=0}^\infty\int_{s=0}^a \left(\frac{\varphi(s)^{2/N}}{2}-\frac{\varphi(s)^{k/N}}{k}\right) s^{N-1}a^{2-N}\,ds d\bar a\,.
\end{align*}

\noindent\textbf{Step 6 (Convexity by choice of regimes)}:
Finally, we make use of the inequalities
\begin{align}
\frac{z^{1-m}}{m-1}  -\frac{z^{1-m_c}}{m_c-1}
&\ge \frac{1}{m-1}  -\frac{1}{m_c-1}\qquad z\ge 0\,, \quad m\ge m_c\,,
\label{zbound1}\\
\frac{z^{2/N}}{2}  -\frac{z^{k/N}}{k}
&\ge \frac{1}{2}  -\frac{1}{k}      \hspace{2.5cm}z\ge 0\,, \quad k<0\,.
\label{zbound2}
\end{align}
to bound $\mF[\rho]$ from below once more, following the estimate in Step 5.
From the formulation \eqref{Fsstatetransport} in Step 1, we conclude
$$
\frac{1}{N\sigma_N}\mF[\rho] \ge \frac{1}{N\sigma_N}\mF[\bar\rho]\,.
$$
\noindent\textbf{Step 7 (Equality cases)}:
Equality in Jensen's inequality \eqref{eq:weightedJensen} arises if and only if the derivative of the transport map, $\psi''$, is a constant function, i.e. when $\rho$ is a dilation of $\bar \rho$. In agreement with this, equality in \eqref{zbound1}--\eqref{zbound2} is realised if and only if $z=1$, that is, $\rho = \bar \rho$. We conclude that equality in the functional inequality in Theorem~\ref{thm:main1} is realised if and only if $\rho = \bar \rho$, unless $m=m_c$ and $\mur=0$, in which case the equality cases correspond to dilations of $\bar\rho$.
\end{proof}

The proof of Theorem~\ref{thm:main1} for the harmonic case $k=2-N$ is similar in strategy to the sub-harmonic case, but simpler in terms of calculations. This is why we are not breaking the argument down into steps this time.
\begin{proof}[Proof of Theorem~\ref{thm:main1} for $k=2-N$]
Similar to Step 1 above, the interaction energy of the stationary state can be rewritten as follows the characterisation~\eqref{eq:StStcharacterization radial} provided in Lemma~\ref{lem:charsstatessubNewtonian}:
\begin{align*}
 &\frac{\sigma_N}{2-N} \int_{a = 0}^{\infty} \bar\rho(a) M_{\bar\rho}(a) a\,  da\\
 &\quad= \frac{N \sigma_N}{2-N} \int_{a = 0}^{\infty}\left(\int_{b = 0}^{a} b^{N-1} \,db\right)a^{1-N} \bar\rho(a) M_{\bar\rho}(a) \,  da\\
 &\quad=\frac{N \sigma_N}{2-N} \int_{b = 0}^{\infty}\left(\int_{a = b}^{\infty} a^{1-N} \bar\rho(a) M_{\bar\rho}(a) \,  da \right)b^{N-1} \,db\\
&\quad=\frac{N \sigma_N}{2-N} \int_{b = 0}^{\infty} \bar \rho(b)^m b^{N-1} \,db
 -\mur \frac{N \sigma_N}{2-N} \int_{b = 0}^{\infty}\left(\int_{a = b}^{\infty} a^{2-N}  \,  d\bar a \right)b^{N-1} \,db\\
&\quad= \frac{N \sigma_N}{2-N} \int_{a = 0}^{\infty} \bar \rho(a)^m a^{N-1} \,da
 -\mur \frac{N \sigma_N}{2-N} \int_{a = 0}^{\infty}\left(\int_{b = 0}^{a}b^{N-1} \,db \right) a^{2-N}  \,  d\bar a\\
  &\quad=\frac{N \sigma_N}{2-N} \int_{a = 0}^{\infty} \bar \rho(a)^m a^{N-1} \,da
 -\mur \frac{ \sigma_N}{2-N} \int_{a = 0}^{\infty} a^{2}  \,  d\bar a\,.
 \end{align*}
Substituting into the expression for $\mF$ derived in Proposition~\ref{prop:sec2}, we obtain the following expression for the free energy of the stationary states:
\begin{align*}
    \frac{1}{N\sigma_N}\mF[\bar\rho]
    = \left(\frac{1}{N(m-1)}+\frac{1}{2-N}\right)\int_{a=0}^\infty\bar\rho(a)^m a^{N-1}\,da
    + \frac{\mur}{N}\left(\frac{1}{2}-\frac{1}{2-N}\right)\int_{a=0}^\infty a^2 \,d\bar a\,.
\end{align*}
Next, we estimate the interaction term of the free energy for $\rho$ using \eqref{Jensen1} for $k=2-N$,
\begin{align*}
\int_{r = 0}^{\infty} \rho(r) M_\rho(r) r\,  dr
=& \int_{a = 0}^{\infty}    M_{\bar \rho}(a)   \left(  \psi'(a) \right)^{2-N}  \,d\bar a\\
\leq & N\int_{a = 0}^{\infty}\int_{b = 0}^a       M_{\bar \rho}(a)  \varphi(b)^{(2-N)/N} b^{N-1} a^{2-2N} \,   db d\bar a \\
=& N \int_{b = 0}^{\infty}    \varphi(b)^{(2-N)/N}    \left\{  \int_{a = b}^{\infty}     M_{\bar \rho}(a) a^{2-2N}  \, d\bar a\right\} b^{N-1}\, db \,.
\end{align*}
Therefore, we have
\begin{align*}
\int_{r = 0}^{\infty} \rho(r) M_\rho(r) r\,  dr
=& N \int_{b = 0}^{\infty}  \varphi(b)^{(2-N)/N} \bar \rho(b)^m b^{N-1}\, db \\
&\qquad
- N\mur \int_{b = 0}^{\infty}    \varphi(b)^{(2-N)/N}    \left\{  \int_{a = b}^{\infty} a^{2-N}  \, d\bar a\right\} b^{N-1}\, db 
\end{align*}
since $M_\rho(r)=M_{\bar \rho}(a)$ and where we used the characterisation~\eqref{eq:StStcharacterization radial} of stationary states provided in Lemma \ref{lem:charsstatessubNewtonian}. Estimating the confinement term as in \eqref{confinementJensen}, we obtain the following estimate on the free energy as given in Proposition~\ref{prop:sec2}:
\begin{align*}
 \frac{1}{N\sigma_N}\mF[\rho] \geq &
 \int_{a=0}^{\infty} \left(\frac{\varphi(a)^{1-m}}{N(m-1)} + \frac{\varphi(a)^{(2-N)/N}}{2-N} \right) \bar \rho(a )^m  a^{N-1}\, da\\
 &+\mur \int_{a=0}^\infty\int_{b=0}^a \left(\frac{\varphi(b)^{2/N}}{2}-\frac{\varphi(b)^{(2-N)/N}}{2-N}\right) a^{2-N}b^{N-1}\, db d\bar a\,.
\end{align*}
We conclude as before using \eqref{zbound1}--\eqref{zbound2}.
\end{proof}

Theorem \ref{thm:main1} directly implies Theorem~\ref{thm:main2}:

\begin{proof}[Proof of Theorem~\ref{thm:main2}]
  Assume there are two radial stationary states to equation \eqref{eq:KSrescradial} with the same mass $M$: $\bar \rho_1, \bar \rho_2 \in \mY_M^*$. Then Theorem \ref{thm:main1} implies that $\mF[\bar\rho_1]=\mF[\bar \rho_2]$, and so $\bar \rho_1=\bar\rho_2$ (up to dilations if $m=m_c$, $\mur=0$ and $M=M_c$).
\end{proof}

\appendix
\section{Properties of Hypergeometric Functions}\label{sec:hypergeo}
In this work, we are making frequent use of the fact that the Riesz potential of a radial function can be expressend in terms of the Gauss Hypergeometric Function,
\begin{equation}\label{hyposeries}
F(a,b;c;z) := \sum_{n = 0}^{\infty} \dfrac{(a)_n(b)_n}{(c)_n} \dfrac{z^n}{n!}
\end{equation}
which we define for $z \in (-1,1)$, with parameters $a,b,c$ being positive. Here $(q)_n$ denotes the Pochhammer symbol,
\begin{equation*}
(q)_n :=
\begin{cases}
q(q+1)\cdots(q+n-1)\,, &\text{if}\, n>0\,,\\
1\,, & \text{if}\, n=0\,.
\end{cases}
\end{equation*}
In our context, the following analytical continuation allows to establish the link with the Riesz potential,
\begin{equation*}
F(a,b;c;z):=\frac{\Gamma(c)}{\Gamma(b)\Gamma(c-b)}\int_{0}^1(1-zt)^{-a}(1-t)^{c-b-1}t^{b-1}\,dt,
\end{equation*}
Notice that $F(a,b,c,0)=1$ and $F$ is increasing with respect to $z\in(-1,1)$.
Moreover, if $c>1$, $b>1$ and $c>a+b$, the limit as $z\uparrow 1$ is finite and it takes the value
\begin{equation*}\label{hyperlimit}
\frac{\Gamma(c)\Gamma(c-a-b)}{\Gamma(c-a)\Gamma(c-b)},
\end{equation*}
see \cite[\S 9.3]{Leb}.
To simplify notation, let us define
\begin{equation}\label{defH}
 H(a,b;c;z):= \frac{\Gamma(b)\Gamma(c-b)}{\Gamma(c)} F(a,b;c;z)
 = \int_{0}^1(1-zt)^{-a}(1-t)^{c-b-1}t^{b-1}\,dt\, .
\end{equation}
 We will also make use of some elementary relations. Firstly, the derivative of $F$ in $z$ is given by \cite[15.2.1]{Abramowitz}
 \begin{equation}\label{hyperdiff}
  \frac{d}{dz} F(a,b;c;z)=\frac{ab}{c}F(a+1,b+1,c+1,z)\, .
 \end{equation}
Further, the following quadratic transformation holds true for hypergemetric functions \cite[Formula 15.3.17]{Abramowitz}:
\begin{equation}\label{quadratictrans}
F(a,b;2b;z)=\left(\frac{1}{2}+\frac{1}{2}\sqrt{1-z}\right)^{-2a} F\left(a,a-b+\frac{1}{2};b+\frac{1}{2};\left(\frac{1-\sqrt{1-z}}{1+\sqrt{1-z}}\right)^2\right)\, .
\end{equation}
Finally, we will make use of the following two identities \cite[15.2.18 and 15.2.17]{Abramowitz},
\begin{equation}\label{hypershift1}
 (c-a-b)F(a,b;c;z)-(c-a)F(a-1,b;c;z)+b(1-z)F(a,b+1;c;z)=0\, ,
\end{equation}
and
\begin{equation}\label{hypershift2}
 (c-a-1)F(a,b;c;z)+aF(a+1,b;c;z)-(c-1)F(a,b;c-1;z)=0\, .
\end{equation}

\section{Proof of Lemma \ref{lem:comparison from below}}\label{sec:lem6}


In this appendix we give a complete proof of Lemma~\ref{lem:comparison from below}. The case $N=2$ will be treated separately, and we present here two different proofs. The first one is in the same spirit as for higher dimensions and uses the integral representation \eqref{defH} to motivate inequality \eqref{eq:comparison from below}, a simple consequence of Jensen's inequality. The second proof is much shorter and follows from a simple convexity argument, however, it cannot be generalised to higher dimensions up to our knowledge.

We begin by recalling the statement of Lemma~\ref{lem:comparison from below}:
\begin{lemma*} 
Assume $N\geq 2$ and $k\in (-N,2-N)$. The following inequality holds true for any $(t,c)\in(0,1)^2$:
\begin{equation} \label{eq:comparison from below}
\frac{\vt(t)}{k}\geq \alpha(c) +\beta(c) \left(1 - t^N\right)^{k/N}\, .
\end{equation}
with the two factors given by
\[\alpha(c):=\frac{\vt(c)}{k} + \frac{1}{k^2} c^{1-N} \left(1 - c^N\right) \vt'(c) \leq 0\, ,\]
and
\[\beta(c):=-\frac{1}{k^2} c^{1-N} \left(1 - c^N\right)^{1 - k/N} \vt'(c) \leq 0\, .\]
\end{lemma*}

\subsection{Two-Dimensional Setting}

\begin{proof}[Proof 1 of Lemma \ref{lem:comparison from below} in dimension $N=2$]
We have $d_2=2\pi$, and so
\begin{align}\label{hyper1}
\vartheta_{k,2}(t)  &= 2\pi F\left( -\dfrac k2,-\dfrac{k}2, 1,t^2  \right)
=\Gamma_k H\left( -\dfrac k2,-\dfrac{k}2, 1,t^2  \right)\\
& = \Gamma_k  \int_{u = 0}^1 u^{-k/2 - 1}(1-u)^{k/2} \left( 1 - t^2 u \right)^{k/2}\, du\, ,\notag
\end{align}
where
\begin{align*} \Gamma_k := \frac{2\pi}{\Gamma(-k/2)\Gamma(1+k/2)}\, .
\end{align*}

We aim at decoupling $t$ and $u$, which is a crucial step in the forthcoming estimates. By convexity of $(\cdot)^{k/2}$,
\begin{align*}
\left[ 1 - u + \left(1 - t^2\right)u  \right]^{k/2} & = \left[ \alpha \dfrac{1-u}\alpha + (1-\alpha) \dfrac{\left(1 - t^2\right)u}{(1-\alpha)}  \right]^{k/2}\\
& \leq  \alpha  \left[\dfrac{1-u}\alpha\right]^{k/2} +  (1-\alpha)   \left[  \dfrac{\left(1 - t^2\right)u}{(1-\alpha)}   \right]^{k/2}\, ,
\end{align*}
where the coefficient $\alpha$ is chosen such that equality arises for $t = c$:
\begin{align*}
\dfrac{1 - u}\alpha& = \dfrac{\left(1 - c^2\right)u}{(1-\alpha)}\, , \quad
\alpha = \dfrac{1 - u}{  1 - c^2 u   } \, , \quad
1 - \alpha = \dfrac{\left(1 - c^2\right)u}{   1 - c^2 u   }\, .
\end{align*}
Therefore we have:
\begin{align*}
&\int_{u = 0}^1 u^{-k/2 - 1}(1-u)^{k/2} \left( 1 - t^2 u \right)^{k/2}\, du\\
 & \leq \int_{u = 0}^1 u^{-k/2 - 1}(1-u)^{k/2} \left\{ (1 - u) (1 - c^2u)^{k/2 - 1} + (1 - c^2)^{1 - k/2}u (1 - c^2u)^{k/2 - 1 }  (1 - t^2)^{k/2} \right\}\, du \\
& \leq\int_{u = 0}^1 u^{-k/2 - 1}(1-u)^{k/2 + 1}  (1 - c^2u)^{k/2 - 1} \, du \\
& \qquad + (1 - c^2)^{1 - k/2} (1 - t^2)^{k/2} \int_{u = 0}^1 u^{-k/2 }(1-u)^{k/2 }  (1 - c^2u)^{k/2 - 1} \, du \\
&= H\left( -\dfrac k2+1,-\dfrac{k}2, 2,c^2  \right) + (1 - c^2)^{1 - k/2} (1 - t^2)^{k/2} H\left( -\dfrac k2+1,-\dfrac{k}2+1, 2,c^2  \right)\, .
\end{align*}
To rewrite $H$ in terms of the hypergeometric function $F$, recall that $\Gamma(z+1)/\Gamma(z)=z$ for any $z\in \C$ that is not an integer less or equal to zero, and so we have 
$\Gamma\left(-\frac{k}{2}+1\right)/\Gamma\left(-\frac{k}{2}\right)=-\frac{k}{2}$. Additionally, since $k\in(-2,0)$, we have $k \neq -2n$, $n \in \N_{>0}$, and therefore
$$
\frac{\Gamma\left(2+\frac{k}{2}\right)}{\Gamma\left(1+\frac{k}{2}\right)} = 1+\frac{k}{2}\, .
$$
It follows that 
\begin{align}\label{hyper2}
\vartheta_{k,2}(t)& \leq  \Gamma_k H\left( -\dfrac k2+1,-\dfrac{k}2, 2,c^2  \right)  
 + (1 - c^2)^{1 - k/2} (1 - t^2)^{k/2} \Gamma_k H\left( -\dfrac k2+1,-\dfrac{k}2+1, 2,c^2  \right)\notag\\
& = 2 \pi \left(1 + \dfrac k2\right)  F\left( -\dfrac k2+1,-\dfrac{k}2, 2,c^2  \right) \notag\\
&\quad - \pi k (1 - c^2)^{1 - k/2} (1 - t^2)^{k/2}  F\left( -\dfrac k2+1,-\dfrac{k}2+1, 2,c^2  \right)
\, .
\end{align}
We have on the one hand from \eqref{hyperdiff}, 
\begin{equation}\label{eq:comparison below n2 A}
\vartheta_{k,2}'(c) =  \pi c k^2 F\left(-\dfrac k 2+1,-\dfrac k 2+1,2 ,c^2 \right)\, .
\end{equation}
On the other hand, we deduce from \eqref{hypershift1} and \eqref{hypershift2}, 
\begin{align*} (1 + k) F\left(-\dfrac k 2+1,-\dfrac k 2,2 ,c^2 \right) & =   \dfrac k2 (1 - c^2) F\left(-\dfrac k 2+1,-\dfrac k 2+1,2 ,c^2 \right)\\
&\quad+ \left(1+\dfrac k2\right) F\left(-\dfrac k 2,-\dfrac k 2,2 ,c^2 \right) \,, \\
\left(1+\dfrac k2\right) F\left(-\dfrac k 2,-\dfrac k 2,2 ,c^2 \right)  &= \dfrac k2 F \left(-\dfrac k 2+1,-\dfrac k 2,2 ,c^2 \right) +  F\left(-\dfrac k 2,-\dfrac k 2,1 ,c^2 \right)\, , 
\end{align*}
that
\[
\left( 1+ \dfrac k2\right) F\left(-\dfrac k 2+1,-\dfrac k 2,2 ,c^2 \right) = 
\dfrac k2 (1 - c^2) F\left(-\dfrac k 2+1,-\dfrac k 2+1,2 ,c^2 \right) + F\left(-\dfrac k 2,-\dfrac k 2,1 ,c^2 \right)\, .
\]
Combining the above and \eqref{hyper1}, the last identity rewrites as
\begin{equation}\label{eq:comparison below n2 B}
\left( 1+ \dfrac k2\right) F\left(-\dfrac k 2+1,-\dfrac k 2,2 ,c^2 \right) = \frac{ (1 - c^2)}{2\pi c k} \vartheta_{k,2}'(c) + \frac{1}{2\pi} \vartheta_{k,2}(c)\, .
\end{equation}
The two relations \eqref{eq:comparison below n2 A} and \eqref{eq:comparison below n2 B} applied to \eqref{hyper2} complete the proof of \eqref{eq:comparison from below} in dimension $N=2$.
\end{proof}


Ultimately, we would like to use Lemma~\ref{lem:comparison from below} to prove our main result Theorem~\ref{thm:main1}. In this context, we will apply the convexity estimate in a particular setting: 
For a given stationary state $\bar \rho \in \mY_M^*$ and any radial function $\rho \in \mY_M^*$, we denote by $\psi$ the radial profile of the convex function whose gradient pushes forward the measure $\bar \rho(a)a^{N-1} da$ onto $\rho(r) r^{N-1}dr$: 
$$\psi' \# \left(\bar\rho(a) a^{N-1} da\right) = \rho(r) r^{N-1} dr\,.$$ 
In this special case where $N=2$, $c=b/a$ and $t=\psi'(b)/\psi'(a)$ for $b<a$, Lemma~\ref{lem:comparison from below} can be shown by a simple convexity argument:\\

\begin{proof}[Proof 2 of Lemma~\ref{lem:comparison from below} in dimension $N=2$]
Let $u\in (0,1)$.
Since $(\cdot)^{k/2}$ is convex, we have directly from the definition of a convex function
\begin{align}\label{2Dconvexity}
\left(\psi'(a)^2-\psi'(b)^2u\right)^{k/2}
= &\left((1-u)a^2 \frac{\psi'(a)^2}{a^2}+u(a^2-b^2)\left(\frac{\psi'(a)^2-\psi'(b)^2}{a^2-b^2}\right)\right)^{k/2}\notag\\
\leq &(1-u)a^2 (a^2-b^2u)^{k/2-1}\left(\frac{\psi'(a)^2}{a^2}\right)^{k/2}\notag\\
&+ (a^2-b^2)u(a^2-b^2u)^{k/2-1}\left(\frac{\psi'(a)^2-\psi'(b)^2}{a^2-b^2}\right)^{k/2}
\end{align}
Writing $\vartheta_{k,2}(t)$ for $t=\psi'(b)/\psi'(a)$ explicitly in terms of the hypergeometric function $F$, we have from \eqref{hyper1} and (\ref{2Dconvexity}):
\begin{align*}
 &\vartheta_{k,2}(t)=\Gamma_k\psi'(a)^{-k} \int_0^1 u^{-k/2-1}(1-u)^{k/2}\left(\psi'(a)^2-\psi'(b)^2 u\right)^{k/2}\,du\\
 &\leq \Gamma_k  \int_0^1 u^{-k/2-1}(1-u)^{k/2+1}\left(1-\frac{b^2}{a^2}u\right)^{k/2-1}\,du\\
 &\quad+\Gamma_k \psi'(a)^{-k} \int_0^1 u^{-k/2}(1-u)^{k/2}(a^2-b^2)^{1-k/2}a^{k-2}\left(1-\frac{b^2}{a^2}u\right)^{k/2-1}\left(\psi'(a)^2-\psi'(b)^2 \right)^{k/2}\,du\\
 &=\Gamma_k  \int_0^1 u^{-k/2-1}(1-u)^{k/2}\left(1-\frac{b^2}{a^2}u\right)^{k/2-1}\,du\\
 &\quad-\Gamma_k  \int_0^1 u^{-k/2}(1-u)^{k/2}\left(1-\frac{b^2}{a^2}u\right)^{k/2-1}\,du\\
 &\quad+(a^2-b^2)^{1-k/2}a^{k-2}\left(\Gamma_k \int_0^1 u^{-k/2}(1-u)^{k/2}\left(1-\frac{b^2}{a^2}u\right)^{k/2-1}\,du\right) \left(1-t^2 \right)^{k/2}.
\end{align*}
Note that 
\begin{align*}
\vartheta_{k,2}\left(\frac{b}{a}\right)&=\Gamma_k  \int_0^1 u^{-k/2-1}(1-u)^{k/2}\left(1-\frac{b^2}{a^2}u\right)^{k/2}\,du\,,\\
\vartheta_{k,2}'\left(\frac{b}{a}\right)&=- \frac{bk}{a} \left(\Gamma_k  \int_0^1 u^{-k/2}(1-u)^{k/2}\left(1-\frac{b^2}{a^2}u\right)^{k/2-1}\,du\right)\,,
\end{align*}
and hence
\begin{align*}
\vartheta_{k,2}\left(\frac{b}{a}\right)-\frac{b}{ak} \vartheta_{k,2}'\left(\frac{b}{a}\right)&=\Gamma_k  \int_0^1 u^{-k/2-1}(1-u)^{k/2}\left(1-\frac{b^2}{a^2}u\right)^{k/2-1}\,du\,.
\end{align*}
Substituting these expressions and recalling $c=b/a$, we obtain
\begin{align*}
 &\vartheta_{k,2}(t)\\
&\le
 \vartheta_{k,2}\left(\frac{b}{a}\right)-\frac{b}{ak} \vartheta_{k,2}'\left(\frac{b}{a}\right)
+\frac{a}{bk} \vartheta_{k,2}'\left(\frac{b}{a}\right)
 - \frac{a}{bk}(a^2-b^2)^{1-k/2}a^{k-2}\vartheta_{k,2}'\left(\frac{b}{a}\right) \left(1-t^2 \right)^{k/2}\\
&= 
\vartheta_{k,2}\left(c\right)
+\frac{1}{kc}(1-c^2) \vartheta_{k,2}'\left(c\right)
 - \frac{1}{kc}(1-c^2)^{1-k/2}\vartheta_{k,2}'\left(c\right) \left(1-t^2 \right)^{k/2}\,,
\end{align*}
which concludes the proof.
\end{proof}


\subsection{Higher-Dimensional Relative Convexity}

In higher dimension $N \geq 3$ the proof of Lemma~\ref{lem:comparison from below} becomes more involved. As we are not aware of any suitable inequality involving hypergeometric functions, we argue directly from the representation using series \eqref{hyposeries}.
Further, we will make use of relative convexity properties defined as follows:

\begin{definition}[Relative convexity]\label{def:relconvexity}
Let $g$ and $\varphi$ be $C^2$ functions defined on some interval $I\subset \RR$. We say that $g$ is \emph{convex relatively to} $\varphi$ if and only if the following convexity-like inequality holds true:
\begin{equation*}
\forall (t,c)\in I^2\quad  g(t) \geq \alpha + \beta \varphi(t)\, ,
\end{equation*}
where $\alpha$ and $\beta$ are chosen in order to fulfill zeroth and first-order approximation at $t = c$:
\begin{align*}
& \alpha = g(c) -  \dfrac{g'(c)}{\varphi'(c)} \varphi(c)\, ,\\
& \beta = \dfrac{g'(c)}{\varphi'(c)}\, .
\end{align*}
In other words, the function $g\circ\varphi^{-1}$ is convex.
\end{definition}

A straightforward computation shows that $g\circ\varphi^{-1}$ is convex if and only if the following criterion is valid:
\begin{equation*}
\forall t\in (0,1)\quad g''(t)\geq \dfrac{\varphi''(t)}{\varphi'(t)} g'(t)\, .
\end{equation*}

\begin{proof}[Proof of Lemma \ref{lem:comparison from below} in dimension $N\geq 3$] By Definition \ref{def:relconvexity}, Lemma \ref{lem:comparison from below} states that the function $\vartheta(t)/k$ with $\vartheta(\cdot)=\vartheta_{k,N}(\cdot)$ as defined in \eqref{eq:def function H} is convex relatively to $(1 - t^N)^{k/N}$. 
Or alternatively, the function $g(z)$ defined by
\[g(z) :=  \frac{d_N}{k}  F\left(\bar a,\bar b; \bar c; z  \right)  = \frac{d_N}{k} \sum_{n = 0}^{\infty} \dfrac{(\bar a)_n(\bar b)_n}{(\bar c)_n}\dfrac{z^n}{n!}  \]
with
\[\bar a:= -\dfrac k2\, , \quad \bar b := 1-\frac{k+N}{2}\, , \quad
 \bar c := \frac{N}{2}
\]
is convex relatively to $(1 - z^{N/2})^{k/N}$. This statement is equivalent to the following inequality:
\begin{equation} \label{eq:appendix 1}
z g''(z) \geq \left( \dfrac k2 - 1 + \dfrac{N-k}{2}\dfrac1{1 - z^{N/2}} \right) g'(z)\, .
\end{equation}
Note that here $\bar b >0$ since $k<2-N$, and so all parameters $\bar a, \bar b, \bar c$ are strictly positive. 
We now use the two following properties:
\begin{enumerate}[(i)]
\item the function $g$ is strictly decreasing when $k\in (-N,2-N)$, \label{item:i}
\item we have the following sharp inequality for $t\in (0,1)$: 
\label{item:ii}
\begin{equation*}
\dfrac{N}{1 - t^N} \geq \dfrac{2}{1 - t^2} + \dfrac{N-2}{2}\, .
\end{equation*}
\end{enumerate}  
The first item \eqref{item:i} is obtained from identity \eqref{hyperdiff}:
\begin{equation*}
g'(z) = \frac{d_N}{k} \dfrac{\bar a\bar b }{\bar c} F(\bar a+1,\bar b +1,\bar c+1,z) =  d_N\left(\dfrac{k+N - 2}{2N}\right) F(\bar a+1,\bar b +1,\bar c+1,z)\, .
\end{equation*}
To obtain the second item \eqref{item:ii}, we need to show that $u(t)\geq 0$ for all $t\in (0,1)$, where we define
$$
u(t):=N(1-t^2)-2(1-t^N)-\left(\tfrac{N-2}{2}\right) (1-t^2)(1-t^N)\,.
$$
Note that $u(0)=(N-2)/2$ and $u(1)=0$. It is therefore enough to show that $u'(t)\leq 0$ on $(0,1)$. 
Differentiating, we have
\begin{align*}
   u'(t)&=-2Nt+2Nt^{N-1}+\left(\tfrac{N-2}{2}\right) \left[2t(1-t^N)+Nt^{N-1}(1-t^2)\right] \\
   &=-(N+2)\left(t-\tfrac{N}{2}t^{N-1}+\left(\tfrac{N-2}{2}\right)t^{N+1}\right)\,.
\end{align*}
By Young's inequality,
\begin{equation*}
t^{N-1} \leq \tfrac2 N t^{\theta N/2} + \left(\tfrac{N-2}{N}\right) t^{\nu N/(N-2)}\, , \quad \theta = \dfrac 2N\, , \quad \nu = \dfrac{(N-2)(N+1)}{N}\, .
\end{equation*} 
and so $u'(t)\leq 0$ follows. This concludes the prove of \eqref{item:ii}.\\
Hence, in order to show inequality \eqref{eq:appendix 1} it is enough to prove
\begin{align*} 
& z g''(z) \geq \left( \dfrac k2 - 1 + \dfrac{N-k}{N}\dfrac1{1 - z}  + \dfrac{(N-k)(N-2)}{4N}\right) g'(z)
\end{align*}
thanks to identity (ii).
This is equivalent to the inequality
\begin{align*}
& (1- z)z g''(z) \geq \left( \dfrac{(N+k)(N-2)}{4N} + \dfrac{4N - 2Nk - N^2 + (k+2)N-2k}{4N}z  \right) g'(z)\, .
\end{align*}
Dividing by $z$ and using that $z\leq 1$, it is enough to prove that
\begin{align*}  
& (1- z) g''(z) \geq \left( \dfrac{(N+k)(N-2)}{4N} + \dfrac{4N - 2Nk - N^2 + (k+2)N-2k}{4N}  \right) g'(z)\, ,
\end{align*}
which simplifies to
\begin{align}\label{lem6equiv}
& (1- z) g''(z) \geq \left( 1 - \dfrac k N  \right) g'(z)\, .
\end{align}
We conclude that \eqref{lem6equiv} directly implies Lemma \ref{lem:comparison from below}.
We now examine inequality \eqref{lem6equiv} term by term using the representation by series. To establish a relation between $g'$ and $g''$, we make use of the following identities for the Pochhammer symbol:
\begin{align*}
 (q+1)_n = \frac{q+n}{q}(q)_n\, , \qquad
 (q)_{n+1}=(q+n) (q)_n\, .
\end{align*}
We can then write the left hand side of \eqref{lem6equiv} as
\begin{align*}
 &(1-z)g''(z)\\
 &= (1-z)d_N\left(\dfrac{k+N - 2}{2N}\right)\dfrac{(\bar a+1)(\bar b+1)}{(\bar c+1)}\sum_{n = 0}^{\infty} \dfrac{(\bar a+2)_n(\bar b+2)_n}{(\bar c+2)_n}\dfrac{z^n}{n!}  \\
  &= d_N\left(\dfrac{k+N - 2}{2N}\right)\dfrac{(\bar a+1)(\bar b+1)}{(\bar c+1)}\sum_{n = 0}^{\infty} \dfrac{(\bar a+2)_n(\bar b+2)_n}{(\bar c+2)_n}\dfrac{\left(z^n-z^{n+1}\right)}{n!}  \\
 &= d_N\left(\dfrac{k+N - 2}{2N}\right)\dfrac{(\bar a+1)(\bar b+1)}{(\bar c+1)}\sum_{n = 0}^{\infty}
 \left( \dfrac{(\bar a+1+n)(\bar b +1+n) - n(\bar c+1+n)}{(\bar a+1+n)(\bar b +1+n)}\right) 
 \dfrac{(\bar a+2)_n(\bar b+2)_n}{(\bar c+2)_n}\dfrac{z^n}{n!}  \\
&= d_N\left(\dfrac{k+N - 2}{2N}\right)\sum_{n = 0}^{\infty}
\left( \dfrac{(\bar a+1+n)(\bar b +1+n)}{(\bar c+1+n)} - n
\right) 
\left(\dfrac{(\bar a+1)_n(\bar b+1)_n}{(\bar c+1)_n}\right)\dfrac{z^n}{n!} \, .
\end{align*}
Comparing this expression term by term to the right hand side of \eqref{lem6equiv},
\begin{align*}
  \left( 1 - \dfrac k N  \right) g'(z)
  &=  \left( 1 - \dfrac k N  \right)d_N\left(\dfrac{k+N - 2}{2N}\right)
  \sum_{n = 0}^{\infty} \left(\dfrac{(\bar a+1)_n(\bar b+1)_n}{(\bar c+1)_n}\right)\dfrac{z^n}{n!} \,,
\end{align*}
we need to show that for all $n\geq 0$,
\begin{equation*}
(\bar a+1+n)(\bar b +1+n) \leq \left(1 - \dfrac k N +n\right) (\bar c+1+n)
\end{equation*}
(note that the sign has changed due to division by $k+N-2<0$). Expanding with respect to $n$, this is equivalent to
$$
\left(\bar a + \bar b -\bar c +\frac{k}{N}\right)n + (\bar a +1)(\bar b +1) - (\bar c+1)\left(1-\frac{k}{N}\right)
\leq 0\,.
$$
We claim that the latter holds true, since we have both
\begin{equation*}
\bar a + \bar b -\bar c +\frac{k}{N}= -k + 1 - N + \dfrac kN = \dfrac{(N+k)(1 - N)}{N}<0\, ,
\end{equation*}
and
\begin{align*}
 (\bar a +1)(\bar b +1) - (\bar c+1)\left(1-\frac{k}{N}\right)
  &= \dfrac{k(N+k)}4 + 1 - k - N + \dfrac kN \\
 & = (N+k)\left( \dfrac k4 + \dfrac 1N - 1 \right) <0\, .
\end{align*}
Finally, the fact that $\beta(c)\leq 0$ follows directly from $\vt'(c)\geq 0$. The sign of $\alpha(c)$ is a consequence of the convexity inequality that we showed above. More precisely, as remarked above, inequality \eqref{eq:comparison from below} is equivalent to 
$$
g''(c)\geq \dfrac{\varphi''(c)}{\varphi'(c)} g'(c) \qquad \forall c\in (0,1)\,.
$$
Therefore, differentiating $\alpha = g(c) -  \dfrac{g'(c)}{\varphi'(c)} \varphi(c)$ with respect to $c$, we have
\begin{align*}
    \alpha'(c)
   = \frac{\varphi(c)}{\varphi'(c)^2} \left(-g''(c)\varphi'(c)+g'(c)\varphi''(c)\right) \leq 0\,.
\end{align*}
Together with
$$
\alpha(c)\to \frac{\vt(0)}{k}
=\frac{2^{N-2}\sigma_{N-1}}{k}\frac{\Gamma\left(\frac{N-1}{2}\right)^2}{\Gamma(N-1)}
<0 \quad \text{ as } \quad c\to 0^+\,,
$$
we conclude that $\alpha(c)\le 0$ for all $c\in (0,1)$.
This completes the proof of Lemma \ref{lem:comparison from below}.
\end{proof}


\section*{Acknowledgements}
This project has received funding from the European Research Council (ERC) under the European Union's Horizon 2020 research and innovation programme (grant agreements No 639638 and No. 883363). JAC was partially supported by the EPSRC through grant number EP/P031587/1. FH was partially supported by Caltech's von Karman postdoctoral instructorship. FH is grateful to Camille, Marine and Constance Bichet, and Joachim Schmitz-Justin and Rita Zimmermann for their unbelievable hospitality during the SARS-CoV-2 outbreak.


\medskip

\bibliographystyle{abbrv}
\bibliography{uniquenessbooks}
\end{document}